\documentclass[letterpaper, 10 pt, conference]{ieeeconf}  

\IEEEoverridecommandlockouts                              

\usepackage{graphicx}
\usepackage{amsmath} 
\usepackage{amssymb}  
\usepackage{mathtools}
\usepackage{color,xcolor}

\usepackage{amsthm}
\usepackage{hyperref}
\usepackage{cleveref}
\usepackage{cite}
\usepackage{algorithm}
\usepackage[noend]{algpseudocode}
\usepackage{xspace}

\newcommand{\dist}{d_{\mathcal{G}}}

\definecolor{mint}{RGB}{62, 180, 137}

\newcommand{\nbhd}{\mathcal{N}_{\mathcal{G}}}
\newcommand{\knbhd}{\mathcal{N}_{\mathcal{G}}^\kappa}
\newcommand{\minsv}{\sigma_{min}}

\newcommand{\pc}{\tilde{\Psi}}

\newcommand{\pct}{\Psi}
\newcommand{\dpc}{\tilde{\psi}}
\newcommand{\eps}{\ensuremath{\varepsilon}\xspace}
\newcommand{\del}{\ensuremath{\delta}\xspace}
\newcommand{\cost}{\mathrm{cost}}
\newcommand{\calV}{\ensuremath{\mathcal{V}}\xspace}
\newcommand{\calT}{\ensuremath{\mathcal{T}}\xspace}
\newcommand{\calS}{\ensuremath{\mathcal{S}}\xspace}

\newboolean{isfullversion}
\setboolean{isfullversion}{true}
\newcommand{\fvtest}[2]{\ifthenelse{\boolean{isfullversion}}{#1}{#2}}

\DeclareMathOperator*{\argmin}{arg\,min}

\allowdisplaybreaks

\newtheorem{theorem}{Theorem}
\newtheorem{corollary}{Corollary}
\newtheorem{lemma}{Lemma}
\newtheorem{definition}{Definition}

\newtheorem{assumption}{Assumption}

\title{\LARGE \bf
 Distributed Truncated Predictive Control for Networked Systems under Uncertainty: Stability and Near-Optimality Guarantee
}

\author{Eric Xu$^{2,\natural}$, Soummya Kar$^{1}$, and Guannan Qu$^{1}$
\thanks{*This work was supported by NSF Grants 2154171, 2339112, 2330196, CMU CyLab Seed Funding, C3 AI Institute. }
\thanks{$^{1}$Department of Electrical and Computer Engineering, Carnegie Mellon University. Emails: soummyak@andrew.cmu.edu and gqu@andrew.cmu.edu}%
\thanks{$^{2}$Department of Electrical Engineering and Computer Science, University of California, Berkeley. Email: erx@berkeley.edu.}
\thanks{$^\natural$Work performed while author was at CMU.}
}

\begin{document}

\maketitle
\thispagestyle{empty}
\pagestyle{empty}

\begin{abstract}
We study the problem of distributed online control of networked systems with time-varying cost functions and disturbances, where each node only has local information of the states and forecasts of the costs and disturbances. We develop a distributed truncated predictive control (DTPC) algorithm, where each node solves a ``truncated'' predictive optimal control problem with horizon $k$, but only involving nodes in a $\kappa$-hop neighborhood (ignoring nodes outside). We show that the DTPC algorithm satisfies input-to-state stability (ISS) bounds and has regret decaying exponentially in $k$ and $\kappa$, meaning a short predictive horizon $k$ and a small truncation radius $\kappa$ is sufficient to achieve near-optimal performance. Furthermore, we show that when the future costs and disturbances are not exactly known, the regret has exponentially decaying sensitivity to the forecast errors in terms of predictive horizon, meaning near-term forecast errors play a much more important role than longer-term forecasts.
\end{abstract}

\section{INTRODUCTION}

The control of networked systems has retained great popularity for the past few decades because of its myriad of applications  
\cite{zhao2014design, sinopoli2003distributed, varaiya2013max, li2017dynamical}. One critical challenge in the control of networks is the distributed nature of the decision making problem: the system consists of nodes where each node may only have access to local state information. 
This control problem becomes even more challenging when we consider the time-varying and online nature of the decision-making process with potentially inaccurate predictions into the future.
Thus, research on efficiently controlling networked systems in a distributed and online manner has become especially prevalent. One potential solution to the aforementioned control problem that has gained great traction over the years is distributed Predictive Control (PC) \cite{stewart_cooperative_2010, Jia2002, Alonso2023slsmpc}. Despite its popularity, the analysis of distributed PC methods only shows asymptotic guarantees of stability and robustness whereas performance guarantees, such as the regret and optimality of these algorithms, are lacking.  
Recently, progress towards showing such performance guarantees has been made in the distributed control and PC literatures in parallel. 
  
Much progress has been made in the synthesis of distributed controllers that are optimal relative to their sparse structure \cite{anderson2019system,fattahi2020efficient, Ghai20222regmin, fazelnia2017}. One promising distributed control method for networked systems is the utilization of the \textit{spatial decay} property in the optimal centralized controller, i.e., the control gain between the control action of agent $i$ and the state of agent $j$ decays based on the (graph) distance between the two agents. The exact decay rate is generally problem-dependent \cite{bamieh2002distributed, motee2008optimal, Shin2022sensitivity, Sungho2023nearopt, Zhang2023polit}. Particularly of interest is the decay rate for the problem described in \cite{Sungho2023nearopt}: they show that the truncation of the centralized solution to the finite-horizon linear quadratic cost (LQC) problem to a \textit{$\kappa$-hop distributed controller} (i.e. a controller whose gains vanish between nodes that are greater than a distance $\kappa$ away) has near-optimal performance relative to the centralized solution.

On a different note, the PC literature contains a large body of work on stability and robustness guarantees\cite{MAYNE2000789}, \cite{Bemporad1999} and has recently developed online performance guarantees in both LTI and LTV (Linear Time-Varying) settings \cite{Yu2020predictions, Zhang2021onlinelqr, Lin2021perturb, Lin2022bounded}. 
The critical property established in \cite{Lin2021perturb} is the \textit{temporal decay} between the predictive states generated by PC controllers that have different initial conditions; they leverage this property to determine ISS (Input-to-State Stability) and regret bounds of PC in the LTV setting. In the sequel \cite{Lin2022bounded}, the authors derive a so-called `pipeline' theorem (Theorem 3.3) which produces regret bounds for PC in several different settings under uncertainties.

Bolstered by the aforementioned success of PC and spatial decay-based distributed control, a natural question to ask is: how can we combine the two to attain efficient distributed and online control for networked systems? Directly combining the works in \cite{Sungho2023nearopt}, \cite{Lin2021perturb}, and \cite{Lin2022bounded} is nontrivial because the proofs in \cite{Lin2021perturb} and \cite{Lin2022bounded} rely on the controller being fully centralized: every node observes the global state and makes fully-accurate predictions about the system. In the distributed setting however, each node has access only to its own state information and that of the nodes nearby it, and thus, it cannot make fully-accurate predictions about the whole system. In summary, the problem is: \textit{how can nodes make accurate predictions about the whole system if they only have access to local information? Can we design a distributed PC algorithm that adheres to this information constraint and has comparable performance guarantees to centralized PC?} We deem this problem the \textit{localized information-constrained predictive control problem} (LICPC Problem)

\textbf{Contribution}: To address the LICPC problem, we consider an online optimal control problem for a networked LTI system with time-varying cost functions and disturbances. We develop a graph-truncated PC-style algorithm, where the key ingredient is that each agent solves a localized PC problem using local information from its $\kappa$-hop neighbourhood in the graph. Upon solving the localized PC problem,  each agent deploys its resulting immediate control input. 
We show that this algorithm attains Input to State Stability (ISS) and exponentially decaying dynamic regret in the predictive horizon variable $k$ and the decentralization factor $\kappa$, meaning that for sufficient choices of $k$ and $\kappa$, our online algorithm produces controllers that enjoy stability and near-optimal regret guarantees. The key technical contribution underlying these guarantees is the spatial decay property exhibited by the centralized PC controller. Leveraging the spatial decay property, we show that agents are effectively able to accurately predict the whole system's state using information within its $\kappa$-hop, thereby solving the LICPC problem. 
Also because of these accurate-enough predictions, our algorithm obtains the desired stability and regret properties. 

Furthermore, we develop an uncertainty-aware version of this distributed PC-like algorithm, where the future forecasts for cost functions and disturbances may contain error. We derive ISS and regret bounds for the uncertainty-aware algorithm which mirror the bounds enjoyed by the aforementioned PC-like algorithm assuming perfect forecasts. Critically, we show that the regret's sensitivity to forecast error is exponentially decaying in the prediction horizon, meaning near-term forecasts are much more important than longer term forecasts. 

We note that our analysis focuses solely on LTI systems because it facilitates our sensitivity analysis of the PC problem which allows us to establish the desired spatial decay property. However, we conjecture that stability and regret bounds for our algorithm in the LTV setting are achievable with further analysis similar to \cite{Lin2021perturb}. 

\textbf{Notation}: We denote the cardinality of a set $S$ by $\lvert S\rvert$. The set of non-negative integers is $\mathbb{Z}_{+}$, the set of reals is $\mathbb{R}$, and the set of $m\times n$ matrices is $\mathbb{R}^{m\times n}$. Furthermore, we denote by $[T] := \{0,1,...,T\}$. The operator norm of a matrix $A$ is denoted as $\lVert A \rVert$ and its minimum singular value as $\minsv(A)$. The 2-norm of a vector $v$ is $\lVert v\rVert$. We further denote $A \succ (\succeq) ~0$ to mean that $A$ is a positive (semi)definite matrix and for a matrix $B$, the notation $A \succ B$ means that $A-B \succ 0$. The Moore-Penrose pseudo-inverse of a matrix $M$ is denoted $M^\dagger$. We define the notation $v_{0:T}$ for $v_t$ indexed by $t \in [T]$ as the set: $\{v_0, v_1,...,v_T\}$. For a matrix $A$ we denote the sub-block indexed by $i, j \in \mathcal{V}$, where $\mathcal{V}$ is an index set, as $A[i,j]$; hence, $A$ can be written as a sequence of its sub-blocks: $A = (A[i,j])_{i,j \in \mathcal{V}}$ (similar for vectors).
 We use $O(\cdot)$, $o(\cdot)$, and $\Theta(\cdot)$ as big-$O$, little-$o$, and big-theta notations respectively.

\section{Problem Setup and Preliminaries}
\subsection{Problem Setting}
We consider a graph $\mathcal{G} := \{\mathcal{V}, \mathcal{E}\}$ with $N$ nodes, where $\mathcal{V}$ is the set of nodes/agents, and $\mathcal{E}$ is the set of edges. Let $\nbhd[i] := \{i\} \cup \{j \in \mathcal{V}: (i,j) \in \mathcal{E}\}$ be the neighbourhood of node $i$ and $\dist(i,j)$ be the shortest path distance between nodes $i$ and $j$ based on the graph $\mathcal{G}$. We denote the $\kappa$-hop neighbourhood of a node $i$ as $\nbhd^\kappa[i] := \{j \in \mathcal{V}: \dist(i,j) \leq \kappa\}$ and its boundary as $\partial\nbhd^\kappa[i] := \{j \in \mathcal{V}: \dist(i,j) = \kappa\}$. 

Given the graph $\mathcal{G}$ we can write the LTI dynamics with disturbance $w$ at a node $i$ as
\begin{equation}\label{eq:LTInode}x_{t+1}[i] = \sum_{j \in \nbhd[i]}A[i,j]x_t[j] + B[i,j]u_t[j] + w_t[i].\end{equation}
Here, we take $x_t[i]\in\mathbb{R}^{n_{x_i}}$, $u_t[i] \in \mathbb{R}^{n_{u_i}}$, and $w_t[i] \in \mathbb{R}^{n_{x_i}}$ to be the state, control action, and disturbance at node $i$ and time $t \in \mathbb{Z}_{+}$ respectively; furthermore, $A[i,j] \in \mathbb{R}^{n_{x_i} \times n_{x_j}}$ and $B[i,j] \in \mathbb{R}^{n_{x_i}\times n_{u_j}}$. 
Naturally, we can stack variables across nodes and rewrite \eqref{eq:LTInode} as:
\begin{equation}\label{eq:LTIcent}
    x_{t+1} = Ax_t + Bu_t + w_t,
\end{equation}
where $x_t = (x_t[i])_{i \in \mathcal{V}}$, $u_t = (u_t[i])_{i \in \mathcal{V}}$, and $w_t = (w_t[i])_{i \in \mathcal{V}}$ are the centralized state, control input, and disturbance. Note, when we define $A = (A[i,j])_{i,j \in \mathcal{V}}$ and $B = (B[i,j])_{i,j \in \mathcal{V}}$, we remark that these matrices are networked in the sense that for any pair of nodes $(i,j) \not\in \mathcal{E}$, the submatrices $A[i,j]$ and $B[i,j]$ are 0, i.e., they respect the graph topology of $\mathcal{G}$.

With the above networked LTI system, we define the networked online control problem:
\begin{align}\label{eq:nc_prob}
    \min_{x_{0:T},u_{0:T-1}}&\sum_{i \in \mathcal{V}} \left(\sum_{t=0}^{T}f_t[i]( x_t[i])  + \sum_{t=0}^{T-1} c_{t+1}[i](u_{t}[i])\right), \nonumber \\
\mathrm{s.t.} \: \: &x_{t+1} = Ax_t + Bu_t + w_t, \quad  \forall t \in [T-1], \nonumber \\
&x_0 = \bar{x}, 
\end{align}
for which the time varying cost functions $f_t[i]$ and $c_t[i]$ are decentralized. We can aggregate these costs into centralized ones through summation, i.e., $f_t(x_t) := \sum_{i \in \mathcal{V}}f_t[i](x_t[i])$ and similarly for $c_t(u_{t-1})$. 

Given the dynamical system from before, we introduce two standard concepts:
\begin{definition}\label{def:stab}
    The system $(A,B)$ is $(L,\gamma)$-stabilizable if there exists a controller $K$ such that for $L > 1$, $\gamma \in (0,1)$, and $\lVert K \rVert \leq L$, we have that $\lVert (A-BK)^t\rVert \leq L\gamma^t$ for all $t \in \mathbb{Z}_{+}$.
\end{definition}
We also define the stronger concept of the controllability index of a system $(A,B)$:
\begin{definition}\label{def:cntrlidx}
    Given that $A$ is an $n\times n$ matrix and the system $(A,B)$ is controllable, there exists a positive integer $d \leq n$ such that the reduced controllability matrix
    $\mathcal{C}^{d} := \begin{bmatrix}
        B & AB & \cdots & A^{d-1}B
    \end{bmatrix}$
    is of full row rank.
\end{definition}

\textbf{Exact exogenous forecast vs uncertain exogenous forecast.} The cost functions  $f_t[i], c_{t+1}[i]$ and disturbance $w_t[i]$ for agent $i$ are time varying, and we assume that forecasts\footnote{Throughout this paper, we use ``forecast'' to refer to forecasts of exogenous quantities like the cost functions and disturbance in the future. We use ``prediction'' to refer to the predicted states in predictive control. } are available to each agent for their future costs and disturbances. We will consider two types of forecast. For exact forecast, at time $t$ each agent $i$ knows exactly  $f_{t+\tau}[i], c_{t+1+\tau}[i]$ and disturbance $w_{t+\tau}[i]$ for the future $k$ steps $\tau = 0,\ldots, k-1$. For the uncertain forecast case, each agent only have inexact values for the future cost functions and disturbances up to $k$ steps. To represent such inexact forecast, we let our ground truth disturbances and costs be parameterized by $\theta^*_t = (\theta^*_t[i])_{i\in\calV}$, i.e. $f_t(x_t,\theta_t^*) = \sum_{i\in\calV}f_t[i](x_t[i],\theta_t^*[i])$ and similarly for $c_{t+1}(u_t,\theta_t^*)$, and $w_t(\theta_t^*) = (w_t[i](\theta_t^*[i]))_{i\in\calV}$. Hence, the networked control problem in \eqref{eq:nc_prob} can be rewritten as 
\begin{align}\label{eq:nc_prob_unc}
    \min_{x_{0:T}, u_{0:T-1}}&\sum_{t=0}^{T}f_t(x_t,\theta_t^*) + \sum_{t=0}^{T-1}c_{t+1}(u_t,\theta_t^*), \nonumber  \\
    \mathrm{s.t.} \;\; &x_{t+1} = Ax_t + Bu_t + w_t(\theta_t^*), \: \: \forall t \in [T-1] \nonumber \\
    &x_0 = \bar{x},  
\end{align}
Under such parameterization, each agent $i$ knows the functional form of $f_t[i]$,$c_{t+1}[i]$,$w_t[i]$ but does not have exact values of $\theta_t^*$. Instead, at time $t$, it will only have inexact forecasts of $\theta_{t+\tau}^*$ for $\tau=0,\ldots, k-1$.\footnote{Notation-wise, in the exact forecast case, $f_t[i](x_t[i]),c_{t+1}[i](u_t[i]),w_t[i]$ will not be parameterized by $\theta_t[i]$. In the uncertain forecast case, $f_t[i](x_t[i],\theta_t[i])$, $c_{t+1}[i](u_t[i],\theta_t[i])$, $w_t[i](\theta_t[i])$ will be parameterized by $\theta_t[i]$. This should not cause any confusion as the context of exact/uncertain forecast will be made clear.}

\section{Background}
\subsection{Centralized Predictive Control}
Centralized Predictive Control (PC) is an online algorithm that can be used to solve the finite horizon problem in \eqref{eq:nc_prob}. We first describe PC in the exact forecast case (each agent $i$ knows exactly  $f_{t+\tau}[i], c_{t+1+\tau}[i]$, $w_{t+\tau}[i]$ for the future $k$ steps $\tau = 0,\ldots, k-1$). In particular, we have that at time $t$, the controller observes $k$ (the prediction horizon) information tuples $I_{t:t+k-1}$ where each information tuple is defined as $I_t := (A,B,w_t,f_t,c_{t+1})$ and solves $\pc_t^k(x_t,w_{t:t+k-1};F)$, which is a finite-horizon optimal control problem defined as: {\small
\begin{align}\label{eq:cent_mpc}
    &\pc_t^{k}(x,\zeta; F) := \\
    \argmin_{y_{0:k},v_{0:k-1}} &\sum_{\tau = 0}^{k-1}f_{t+\tau}(y_\tau) + \sum_{\tau=1}^{k}c_{t+\tau}(v_{\tau-1}) + F(y_k), \nonumber\\
    \mathrm{s.t.} \: \: &y_{\tau + 1} = Ay_\tau + Bv_\tau + \zeta_\tau, \quad \forall \tau \in [k-1], \nonumber \\
    & y_0 = x \nonumber , 
\end{align} }where $\zeta \in (\mathbb{R}^{n})^{k}$ is a sequence of $k$ disturbances indexed from $0$ to $k-1$ and $F:\mathbb{R}^{n} \mapsto \mathbb{R}$ is a terminal cost regularizing the final predictive state. At each predictive time step $\tau$, the predictive state and control are $y_\tau \in \mathbb{R}^{n}$ and $v_\tau \in \mathbb{R}^{m}$ respectively. Abusing notation, we will often write $\pc_t^p(x,\zeta)$ for when the terminal cost is either $F(\cdot)$ or $f_T$ and $\pc_t^p(x;\cdot) = \pc_t^p(x,\zeta;\cdot)$ when the disturbances are unambiguous. The algorithm is described in Algorithm \ref{alg:cpc}.
\begin{algorithm}\caption{Centralized Predictive Control ($PC_k$) \cite{rawlings2017model}}\label{alg:cpc}
    \begin{algorithmic}
        \For{$t = 0,1,...,T-k-1$}
            \State Observe $x_t$ and information tuples $I_{t:t+k-1}$. 
            \State Solve $\pc_t^k(x_t,w_{t:t+k-1};F)$ and apply $u_t = v_0$.
        \EndFor
        \State At $t = T-k$, observe $x_{t}$ and information tuple $I_{t:T-1}$.
        \State Solve $\pc_{t}^{T-t}(x_{t},w_{t:T-1};f_T)$ and apply $u_{t:T-1} = v_{0:k-1}$.
    \end{algorithmic}
\end{algorithm}

In the uncertain forecast case, we assume that at time $t$, the controller has access to estimates $\theta_{t:t+k-1|t}$ of the true parameters $\theta^*_{t:t+k-1}$. Note that the predictions $\theta_{t:t+k-1|t}$ can come from any causal estimator so long as its estimates are ``good enough" (c.f. equations \eqref{eq:predsreq1} and \eqref{eq:predsreq2}).
The controller then uses the estimator's predictions $\theta_{t:t+k-1}$ to form information tuples $\hat{I}_{t:t+k-1|t}$ in which each information tuple is defined as $\hat{I}_{t+\tau|t} := (A,B,\hat{w}_{t+\tau|t}, \hat{f}_{t+\tau|t}, \hat{c}_{t+1+\tau|t})$, where $\hat{w}_{t+\tau|t} := w_{t+\tau}(\theta_{t+\tau|t})$ and similarly for the costs $\hat{f}_{t+\tau|t} := f_{t+\tau}(\cdot,\theta_{t+\tau|t})$ and $\hat{c}$. Hence, the uncertainty-aware version of Algorithm \ref{alg:cpc}, denoted as $uPC_k$, has the controller solve the optimization problem in \eqref{eq:cent_mpc} at every timestep $t$ with the costs and disturbances replaced by those in $\hat{I}_{t:t+k-1|t}$.

\subsection{Prior Results}
Before we discuss previous results, we must first define a performance metric. Thus, define the $\cost(\cdot)$ operator which takes as input a control algorithm and returns its cost defined by the objective in \eqref{eq:nc_prob}---that is, the cost of the trajectory $(x_{0:T},u_{0:T-1})$ generated by the control algorithm. Next, we define the optimal control algorithm, $OPT$, which generates the trajectory $(x_{0:T}^*,u_{0:T-1}^*)$ by solving the entire nonlinear program in \eqref{eq:nc_prob_unc}. Thus, we can define the dynamic regret of some control algorithm $ALG$ as $\cost(ALG) - \cost(OPT)$. 

In \cite{Lin2021perturb}, they establish ISS and regret bounds for $PC_k$. These bounds depend on the \textit{temporal decay constant} $\delta_{\mathcal{T}}$ (established in their Lipschitz result (Theorem 3.3 in \cite{Lin2021perturb})): with large enough prediction horizon $k$, the dynamic regret will decay exponentially as $O(\delta_{\mathcal{T}}^k)$. Since then, these results have been extended to the inexact forecast setting in \cite{Lin2022bounded} where they provide a so-called "pipeline" theorem (Theorem 3.3 in \cite{Lin2022bounded}) which determines sufficient conditions for bounding the regret of $uPC_k$. Using the pipeline theorem and several restrictive assumptions, the authors in \cite{Lin2022bounded} produce regret bounds for 3 general settings: LTV systems with uncertain disturbances, LTV systems with uncertain dynamics and uncertain quadratic costs, and nonlinear time-varying systems with uncertain dynamics, constraints, and costs. For the uncertainty setting described in \eqref{eq:nc_prob_unc} in combination with several assumptions, the regret bound according to Theorem 4.2 in \cite{Lin2022bounded} is $
\cost(uPC_k) - \cost(OPT) \leq O\left(\sqrt{T\sum_{n=0}^{k-1}\delta_\calT^n \Phi_n + \delta_\calT^{2k}T^2}+\sum_{n=0}^{k-1}\delta_\calT^n \Phi_n \right)$, where $\cost(OPT)$ is the cost of the offline optimal trajectory, $\pc_{0}^{T}(x_0,w_{0:T-1};f_T)$, and 
\begin{equation}\label{eq:sqerr}\Phi_n := \sum_{t=0}^{T-n}\lVert \theta_{t+n|t}-\theta^*_{t+n}\rVert^2\end{equation} is the \textit{cumulative squared error} of the predictions that are $n$ steps ahead. 

Despite its strong guarantees, it is not immediately clear how to perform $PC_k$ or $uPC_k$ in the networked setting because every node would require access to the global state and information tuples at every time step $t$.  More precisely, each node $i$ would require knowledge of true costs $f_t[j]$ and $c_t[j]$, true disturbances $w_t[j]$ (or the parameter forecasts $\theta_{t+\tau|t}[j]$ in the uncertain forecast case) across all $j$, as well as the full dynamics matrices, $A[j,j']$ and $B[j,j']$ for all $j,j'$, to perform $PC_k$ ($uPC_k$). This requires global information and is hence not practical. Therefore, it would be ideal to have a variant of $PC_k$ ($uPC_k$) for each node that only uses limited information from a local neighborhood. It is nontrivial, however, to design such an algorithm. First of all, a node cannot make predictions of the current and future state of the system with only local information. Secondly, the regret analyses in \cite{Lin2021perturb} and \cite{Lin2022bounded} is facilitated by the centralized setting since it allows for direct comparison between the trajectory generated by $PC_k$ to the optimal offline trajectory. This direct comparison cannot be made in the distributed setting because nodes incur extra error due to their lack of access to the centralized information tuples $I_{t:t+k-1}$.  

\section{Algorithm Design}

The main idea behind the algorithm we will design is that at each time step, every node solves a ``$\kappa$-hop" version of \eqref{eq:cent_mpc} such that they only have access to state information and future information tuples of other nodes that are within a $\kappa$-hop distance of themselves. We first describe the approach in the exact forecast case before generalizing it to the uncertain forecast case.

\subsection{Exact Forecast}
Before explicitly describing the algorithm, we formalize the constraint of local information in the following definition:
\begin{definition}\label{def:ikmat}
    For a matrix $M = (M[j,k])_{j,k\in\mathcal{V}}$ we can define its $(i,\kappa)$-truncation $M^{(i,\kappa)}$ such that
    \begin{equation}\label{eq:iktrunc}M^{(i,\kappa)}[j,k] =\begin{cases}
            M[j,k], \quad j \in \nbhd^\kappa[i], \\
            0, \quad \mathrm{else},
            \end{cases}\end{equation}
\end{definition}
 
Note that this definition can easily be extended to vectors. 
Thus, we can define the distributed counterpart of \eqref{eq:cent_mpc} as: 
{\small
\begin{align}\label{eq:dist_mpc}
    &\dpc_t^k(x,\zeta,\nbhd^\kappa[i];F) := \\
    \argmin_{y_{0:k}, v_{0:k-1}} &\sum_{\tau=0}^{k-1}f_{t+\tau}(y_\tau) + \sum_{\tau=1}^{k}c_{t+\tau}(v_{\tau-1}) + F(y_k), \nonumber \\
    \mathrm{s.t.} \: \: &y_{\tau+1} = A^{(i,\kappa)} y_\tau + B^{(i,\kappa)} v_\tau + \zeta^{(i,\kappa)}_\tau, \quad \tau \in [k-1], \nonumber \\
    &y_0 = x^{(i,\kappa)}. \nonumber 
\end{align}}
Although we write $\dpc_t^k$ as a function of $x$ and $\zeta$, agent $i$ only uses the $(i,\kappa)$-truncated versions of $A$, $B$, $\zeta$, and $x$ to solve \eqref{eq:dist_mpc}. Further, for the costs $f_{t+\tau}(y_\tau) = \sum_{j}f_{t+\tau}[j](y_\tau[j])$, only the costs in the $\kappa$-hop neighbourhood $\knbhd[i]$ are needed since $y_\tau[j] = 0$ for $j$ outside of it. Similarly, $c_{t+\tau}(y_\tau) = \sum_{j}c_{t+\tau}[j](v_{\tau-1}[j])$ only requires costs within the $\kappa$-hop neighbourhood $\nbhd^{\kappa+1}[i]$ since $B^{(i,\kappa)}[j,k]$ is nonzero for $j \in \knbhd[i]$ and $k \in \nbhd[j]$. In other words, to solve (6), agent $i$ only needs to observe $x^{(i,\kappa)}$ and have access to the $k$ localized info tuples: $I_{t:t+k-1}^{(i,\kappa)}$ in which each localized info tuple is $I_{t}^{(i,\kappa)} := (A^{(i,\kappa)}, B^{(i,\kappa)}, w_t^{(i,\kappa)}, f_t^{(i,\kappa)}, c_{t+1}^{(i,\kappa+1)})$ where $f_t^{(i,\kappa)}[j] = f_t[j]$ for $j \in \knbhd[i]$ and $0$ for $j$ outside of $\knbhd[i]$ and $c_{t+1}^{(i,\kappa+1)}$ is defined similarly. 
Based on the definition of truncated optimal control problem \eqref{eq:dist_mpc}, the proposed algorithm can be described as in Algorithm \ref{alg:dtpc}:
\begin{algorithm}\caption{Distributed-Truncated PC ($DTPC_k$)}\label{alg:dtpc}
    \begin{algorithmic}
        \For{$t = 0,1,...,T-1$}
            \For{$i = 1,...,N$}
           \State Agent $i$ observes $x_t^{(i,\kappa)}$ and info tuples $I_{t:t+k-1}^{(i,\kappa)}$ 
		\If{$t < T-k$}
                	\State Solve $\dpc_t^k(x_t, w_{t:t+k-1}, \nbhd^\kappa[i]; F)$ 
		\Else
			\State Solve  $\dpc_t^{T-t}(x_t, w_{t:t+k-1}, \nbhd^\kappa[i]; f_T)$ 
		\EndIf
                \State Collect $v_0[i]$ from the solution and set $u_t[i]=v_0[i]$ 
            \EndFor
            \State Apply $u_t$ to the system. 
        \EndFor
    \end{algorithmic}
\end{algorithm}

$DTPC_k$ answers the first question in the LICPC problem by representing the information constraint via $(i,\kappa)$-truncated versions of the information tuples and state; each agent solves the optimization problem in \eqref{eq:dist_mpc} in which they only have access to $(i,\kappa)$-truncated versions of $A$, $B$, $\zeta$, and $x$. $DTPC_k$ also solves the second question asked by the LICPC problem, as we will show later that node $i$'s predicted states in $DTPC_k$ are accurate despite only observing the current states in $x_t^{(i,\kappa)}$. The underlying reason is that the error incurred from truncating the information at each node is negligible for large enough $\kappa$ due to a spatial decay property (cf. Theorem~\ref{thm:diffdec}). 

\subsection{Extending to Uncertain Forecast}
In the uncertain forecast version of $DTPC_k$, $uDTPC_k$, instead of having access to the true localized info tuples $I_{t:t+k-1}^{(i,\kappa)}$ at each time step $t$, each agent obtains localized predictions $\theta_{t:t+k-1|t}^{(i,\kappa)}$ from the estimator and then uses these estimates to form the predicted localized info tuples $\hat{I}_{t:t+k-1}^{(i,\kappa)}$ (defined in the same way as $\hat{I}_{t:t+k-1}$ but with $(i,\kappa)$-truncated versions of the costs and disturbances). Thus, $uDTPC_k$ proceeds exactly the same as $DTPC_k$ as described in Algorithm \ref{alg:dtpc}, but with the costs and disturbances of the optimization problem $\dpc_t^k$ replaced with those in the predicted localized info tuples $\hat{I}_{t:t+k-1}^{(i,\kappa)}$. 

\begin{algorithm}\caption{Distributed-Truncated PC ($uDTPC_k$)}\label{alg:udtpc}
    \begin{algorithmic}
        \For{$t = 0,1,...,T-1$}
            \For{$i = 1,...,N$}
           \State Agent $i$ observes $x_t^{(i,\kappa)}$ and info tuples $\hat{I}_{t:t+k-1}^{(i,\kappa)}$ 
		\If{$t < T-k$}
                	\State Solve $\dpc_t^k(x_t, \hat{w}_{t:t+k-1}, \nbhd^\kappa[i]; F)$ 
		\Else
			\State Solve  $\dpc_t^{T-t}(x_t, \hat{w}_{t:t+k-1}, \nbhd^\kappa[i]; \hat{f}_T)$ 
		\EndIf
                \State Collect $v_0[i]$ from the solution and set $u_t[i]=v_0[i]$ 
            \EndFor
            \State Apply $u_t$ to the system. 
        \EndFor
    \end{algorithmic}
\end{algorithm}

\section{Main Results}\label{sec:results}
In this section, we present ISS and regret guarantees for $DTPC_k$ and $uDTPC_k$.

\subsection{Exact Forecast}
Before presenting our main result, we state our assumptions.
\begin{assumption}\label{assump:sys}
    The system matrices in \eqref{eq:LTIcent} satisfy the following:
    \begin{enumerate}
        \item $\lVert A \rVert \leq L$, $\lVert B \rVert \leq L$, and $\lVert B^\dagger \rVert \leq L$.
        \item The reduced controllability matrix has minimum singular value $\minsv(\mathcal{C}^d) \geq \sigma$ for $\sigma >0$.
        \item There exists $\kappa_0 < \mathrm{diam}(\mathcal{G})$ (diameter of $\mathcal{G}$). such that for all $i \in \mathcal{V}$ and $\kappa \geq \kappa_0$, the system $(A^{(i,\kappa)},B^{(i,\kappa)})$ is $(L,\gamma)$-stabilizable. 
    \end{enumerate}
\end{assumption}
 Note that the second assumption guarantees existence of an $(L,\gamma)$-stable $K$ in the LTI setting. 
The final assumption essentially says that if we isolate large enough $\kappa$-hop neighbourhoods of nodes from a distributed system, then we expect that these isolated $\kappa$-hop subsystems to be stabilizable if the original distributed system was stabilizable.

We also make the following assumption on the costs:
\begin{assumption}\label{assump:costs}
    The costs are well-conditioned such that:
    \begin{enumerate}
        \item $f_t(\cdot)$ and $c_t(\cdot)$ are $\mu$-strongly convex, $L$-smooth, and twice continuously differentiable for all $t$.
	\item  $F(\cdot)$ is a $\mu$-strongly convex and $L$-smooth K-function (i.e., $F(x) = \beta(\lVert x\rVert)$ where $\beta(\cdot)$ is strictly increasing and $\beta(0) = 0$) 
    and twice continuously differentiable
    \item $f_t(\cdot)$ and $c_t(\cdot)$ are non-negative and $f_t(0) = c_t(0) = 0$ for all $t = 1,...,T$. The same goes for $f_0(\cdot)$.
    \end{enumerate}
\end{assumption}
Note that Assumption \ref{assump:costs} is effectively identical to Assumption 2.1 in \cite{Lin2021perturb}.

Finally, we make one last assumption on the $\tau$-hop neighbourhoods of the graph.
\begin{assumption}\label{assump:d_bands}
    There exists a subexponential function $p(\cdot)$ such that for some distance $d$, we have that 
    \begin{equation}\label{eq:subexp}\lvert j \in \mathcal{V}: \dist(i,j) = d\rvert \leq p(d).\end{equation}
Here, a subexponential function $p$ means $\lim_{d\to\infty} (\log p(d))/d = 0$.
\end{assumption}

Our main contribution is the following ISS and regret bound as stated in the following theorem.
\begin{theorem}\label{thm:perf_guarantees}
    Let the disturbance $w_t$ be uniformly bounded such that $\max_{t \in [T-k]}\sum_{\tau=0}^{k-1}\lVert w_{t+\tau}\rVert \leq D_k$, and let $C := \max(\Omega,\Gamma)$ and $\delta := \max(\delta_{\mathcal{S}}, \delta_{\mathcal{T}})$ where $\Omega$ and $\delta_{\mathcal{T}}$ are as in Lemma \ref{lem:termlpz} and $\Gamma$ and $\delta_{\mathcal{S}}$ are as in Theorem \ref{thm:diffdec}. Under Assumptions \ref{assump:sys}, \ref{assump:costs}, and \ref{assump:d_bands} with constants $\xi = 1-\sqrt{\delta} > 0$ and $L$ from the first two Assumptions, we take $k$ and $\kappa$ such that 
    $$\kappa \geq \max\left(\kappa_0,\frac{\log\frac{(1-\sqrt{\delta})(1-\delta)}{2C^2LN}}{\log\delta}\right), \quad k \geq \frac{2\log\frac{\delta^{5/2}(1-\delta)}{4C^3}}{\log\delta},$$
    which gives the following ISS bound for the system:
    \begin{equation}\label{eq:distISS}
        \lVert x_t\rVert \leq \begin{cases}
            \frac{C}{\xi}(1-\xi)^{\max(0,t-k)}\lVert x_0\rVert + \frac{W}{\xi}, \quad t \leq T-k, \\
            \frac{C^2}{\xi^2}(1-\xi)^{T-2k}\delta^{t+k-T}\lVert x_0\rVert + \frac{2CW}{\xi^2}, \quad \mathrm{else},
        \end{cases}
    \end{equation}
    where $W := \frac{10C^3LN}{(1-\delta)^2}D_k$. Further, we have the regret bound: {\small
    \begin{align}\label{eq:distREG}
        \cost &(DTPC_k) - \cost(OPT) = \nonumber\\
        &O\bigg(\bigg[\left(D_k + \frac{\lVert x_0\rVert + D_k}{\xi^2}\right)^2\delta^{\kappa} \nonumber \\ 
        &\quad + \left(D_k + \frac{\delta^k(\lVert x_0\rVert + D_k)}{\xi}\right)^2\delta^{k}\bigg]T + \eta\lVert x_0\rVert^2\bigg)  
    \end{align} }
    where $\eta = \Theta(\max(\delta^k,\delta^\kappa))$.
\end{theorem}
Theorem \ref{thm:perf_guarantees} says that the regret is exponentially decaying in $\kappa$ and $k$, meaning that if we pick $\kappa$ and $k$ equal to $\Theta(\log T)$, we can attain $o(1)$ dynamic regret. This result parallels that of \cite{Lin2021perturb} with differences. One difference is that in contrast to a single decay constant in \cite{Lin2021perturb}, our bound involves two decay constants $\delta_{\mathcal{S}}$ and $\delta_{\mathcal{T}}$, which correspond to the spatial and temporal decay constants respectively. The spatial decay constant $\delta_{\mathcal{S}}$ depends on the graph and networked structure of the dynamics in \eqref{eq:LTInode}. The temporal decay constant $\delta_{\mathcal{T}}$ is attributed to the Lipschitz property of \eqref{eq:cent_mpc} and its costs' conditioning.
Aside from differences in the constants, the ISS result attains an extra factor of $\xi$ in the last $k$ time steps due to having to explicitly solve \eqref{eq:dist_mpc} at every time step. Furthermore, the regret bound incurs extra error (the $\kappa$ term) from the truncation as expected. 

\subsection{Uncertain Forecast Case}
We require the following assumptions for our main result on $uDTPC_k$ regarding the parameterized class of cost functions and disturbances.

\begin{assumption}\label{assump:uncertainties}
    For the uncertain costs and disturbances, we assume that
    \begin{enumerate}
        \item $f_t(\cdot,\theta_t)$ and $c_t(\cdot,\theta_t)$ are $\mu$-strongly convex, $L$-smooth, twice-continuously differentiable and nonnegative with $f_t(0,\theta_t) = c_t(0,\theta_t) = 0$ for all $t$ and $\theta_t$.
        \item The disturbance $w_t(\theta_t)$ is bounded for all $t$ and $\theta_t$ such that $\lVert w_t(\theta_t)\rVert \leq D$ and lipschitz continuous with constant $L$.
        \item The Hessians $\nabla_{xx}^2f_t$ and $\nabla_{uu}^2c_t$ are lipschitz continuous with respect to the parameter $\theta_t$ for all $t$ with constant $L$.
    \end{enumerate}
\end{assumption}

Under the same setting as in the certainty case and Assumption \ref{assump:uncertainties} we have our main result for the uncertainty setting:

\begin{theorem}\label{thm:uncertaintyresult}
    Let $C := \max(\Omega, \Gamma, \Upsilon)$, where $\Upsilon$ is defined in Theorem \ref{thm:predpertbnd}, and let $\delta$, $\xi$, and $D_k$ be defined as in Theorem \ref{thm:perf_guarantees}. Under Assumptions \ref{assump:sys}, \ref{assump:costs}, \ref{assump:d_bands}, and \ref{assump:uncertainties}, choose $k$ and $\kappa$ such that 
    $$\kappa \geq \max\left(\kappa_0,\frac{\log\frac{\sqrt{\delta}(1-\sqrt{\delta})}{3C^2N}}{\log\delta}\right), \quad k \geq \frac{2\log\frac{\delta(1-\delta)}{6C^3}}{\log\delta}.$$
    Suppose that the predictions satisfy    \begin{align}\label{eq:predsreq1}
        &\sum_{m=0}^{k-1}\sum_{n=0}^{k}\phi_{t-m+n|t-m}\delta^{2n+\frac{m}{2}}\leq \frac{\sqrt{\delta}}{3LC^2}
    \end{align}
    and
    \begin{align}\label{eq:predsreq2}
        \phi_{t+n|t} \leq R(\delta+\epsilon)^{-n}
    \end{align}
    where $\phi_{t+n|t} := \lVert \theta_{t+n|t} - \theta^*_{t+n}\rVert$ is the prediction error, and we are given constants $\epsilon > 0$ and $R > 0$.
    Then $uDTPC_k$ obtains the following ISS bound
\begin{equation}\label{eq:uncertISS}
        \lVert x_t\rVert \leq \begin{cases}
            \frac{C}{\xi}(1-\xi)^{\max(0,t-k)}\lVert x_0\rVert + \frac{W}{\xi}, \quad t\leq T-k, \\
            \frac{C^2}{\xi^2}(1-\xi)^{T-2k}\delta^{t+k-T}\lVert x_0\rVert + \frac{2CW}{\xi^2}, \quad \mathrm{else},
        \end{cases}
    \end{equation}
    where $W = \frac{8LNC^2(\delta+\epsilon)R}{(1-\delta)^2\epsilon}D_k$. $uDTPC_k$ also enjoys the regret bound:
    \begin{align}\label{eq:uncert_reg_bnd}&\cost(uDTPC_k) - \cost(OPT) = \nonumber \\ &O\Bigg(\sqrt{T\sum_{n=0}^{k}\delta_\calT^n\Phi_n + T^2(\delta_\calS^{2\kappa} + \delta_\calT^{2k})}    +\sum_{n=0}^{k}\delta_\calT^n\Phi_n \Bigg)
    \end{align}
    where we recall that $\Phi_n$ is defined as the cumulative square error of the predictions \eqref{eq:sqerr} $n$ steps ahead, and $\delta_\calS$ and $\delta_\calT$ are the spatial and temporal decay constants as defined in Theorem \ref{thm:perf_guarantees}. 
\end{theorem}

Note that the ISS bound \eqref{eq:uncertISS} is nearly identical to \eqref{eq:distISS} with the only difference being in the disturbance term $W$. Under perfect predictions, the regret bound in \eqref{eq:uncert_reg_bnd} reduces $O((\delta^\kappa + \delta^k)T)$ which matches the regret bound in \eqref{eq:distREG} sans the constants absorbed by the big-$O$ notation ($D_k$, $\lVert x_0\rVert$, and $\xi$). As expected, the regret bound \eqref{eq:uncert_reg_bnd} picks up additional terms that are proportional to the cumulative squared error, $\Phi_n$ in \eqref{eq:sqerr}, due to the uncertainty in the estimation. Additionally, our regret bound in \eqref{eq:uncert_reg_bnd} only differs from the regret bound in Theorem 4.2 of \cite{Lin2022bounded} by an extra factor of $\delta^{2\kappa}T^2$ in the square root which can be made small with appropriate $\kappa$. Further note that the regret bound in \eqref{eq:uncert_reg_bnd} and the requirement on the predictions in \eqref{eq:predsreq1} indicate that predictions near the current time $t$ are exponentially more important than those far away from $t$ for the sake of performance.

To demonstrate the power of Theorem \ref{thm:uncertaintyresult}, let us consider two types of forecast errors: 
\begin{itemize}
    \item  We consider the forecast improves over time $t$ but deteriorate exponentially in the look-ahead horizon $n$, i.e. its forecast errors are
\begin{equation}\label{eq:sqrtt}
    \phi^{e}_{t+n|t} = O\left(\frac{1}{\sqrt{t}}(\delta+\eps)^{-n/2}\right)
\end{equation}
for a positive number $\eps$. Under this estimator, we can upper bound the forecast error term in \eqref{eq:uncert_reg_bnd} as:
$$\sum_{n=0}^{k}\delta^n\Phi_n^e \lesssim \sum_{n=0}^{k}\sum_{t=0}^{T-n}\frac{1}{t(1+\eps/\del)^n} = O(\log T).$$
Using the estimator above and choosing $k$ and $\kappa$ equal to $\Theta(\log T)$, $uDTPC_k$ enjoys regret $O(\sqrt{T\log T})$ matching the regret bound in \cite{ManiaCELQC2019}.

\item We consider the forecast errors are constantly bounded over time $t$, but with the same exponential deterioration in look-ahead horizon, i.e.
\begin{equation}\label{eq:constexpdeg}\phi^c_{t+n|t} \leq R(\nu)^{-n}.\end{equation}
If $\theta_t^*$ evolves as a linear dynamical system with gaussian noise, an example of such a predictor which achieves this bounded performance is the kalman filter. If $\nu \geq (\del+\eps)^{1/2}$, it is easy to see that the forecast error term in \eqref{eq:uncert_reg_bnd} is upper bounded as $O(R^2T)$, and so the regret of $uDTPC_k$ under this estimator, with $k$ and $\kappa$ equal to $\Theta(\log T)$, is $O(R^2T)$, meaning that the regret of the algorithm directly depends on the constant performance of the estimation. 

\end{itemize}
   
\section{Proof}
In this section we provide the proof for the exact forecast case (Theorem~\ref{thm:perf_guarantees}). The uncertain forecast case will be handled in Appendix-\ref{subsec:proof_uncertain}. We begin the proof of Theorem \ref{thm:perf_guarantees} with a roadmap.

\textbf{Step 1}: Prove exponential decay between the solution to the centralized problem in \eqref{eq:cent_mpc} versus the solution to \eqref{eq:dist_mpc}. The main idea is that we can express the KKT conditions of both \eqref{eq:cent_mpc} and \eqref{eq:dist_mpc} as the matrix equation $H(z)z = b$ for which we can directly apply the exponential decay of the inverse of $H$ from \cite{Shin2022sensitivity} and \cite{Sungho2023nearopt} to show exponential decay. 
    
\textbf{Step 2}: Show the ISS bound on $DTPC_k$. Here, we will utilize the exponential decay from Step 1 to establish a bound on the difference between the next state produced by $DTPC_k$ and the state that would have been produced by $PC_k$. Then, since $PC_k$ is already ISS, we expect $DTPC_k$ also to be ISS.

\textbf{Step 3}:  We finally prove the regret bound of $DTPC_k$. In order to prove the regret, we will use assumption \ref{assump:costs} and relate the difference between the trajectory from $DTPC_k$ and the offline optimal trajectory $\pc_0^T(x_0,w_{0:T-1};f_T)$ via the ISS bound.


\subsection{Proof of Step 1}
The main goal of Step 1 is to prove the following result: 
\begin{theorem}\label{thm:diffdec}
    Let $q^c$ be the solution vector containing the primal and dual variables for the centralized problem in \eqref{eq:cent_mpc} and let $q^d$ be that of the distributed problem in \eqref{eq:dist_mpc} for some $i \in \mathcal{V}$. Let the prediction horizon be $\ell \leq k$ and the truncation factor be $\kappa \geq \kappa_0$ as in Assumption \ref{assump:sys}. Under the Assumptions \ref{assump:sys}, \ref{assump:costs}, and \ref{assump:d_bands}, we have the following decay result: 
    \begin{equation}\label{eq:diffdec}
        \lVert q^c[i] -q^d[i]\rVert \leq \Gamma(\lVert x\rVert + D_{\ell})\delta_{\mathcal{S}}^\kappa
    \end{equation}
    where the closed brackets $[\cdot]$ will henceforth denote the spatial indexing of some vector or matrix by the network graph $\mathcal{G}$ underlying the dynamics, i.e. $q^c[i]$ means all the entries in $q^c$ corresponding to node $i$ (including all time steps); $D_{\ell} := \max_{t \in [T-1]}\sum_{\tau=0}^{\ell-1}\lVert w_{t+\tau}\rVert$. The constants $\Gamma$ and $\delta_{\mathcal{S}}$ are defined as 
    
    {\small
    $$\begin{aligned}
        \delta_{\mathcal{S}} := \frac{\rho+1}{2}, \quad \Gamma := \frac{4\alpha^2\delta_{\mathcal{S}} L}{(1-\delta_{\mathcal{S}})^2}\left(\sup_{d \in \mathbb{Z_+}}\left(\rho/\delta_{\mathcal{S}}\right)^dp(d)\right)^2p(1),
    \end{aligned}$$} where the function $p(\cdot)$ is from Assumption \ref{assump:d_bands} and constant $L$ is from Assumptions \ref{assump:sys} and \ref{assump:costs}. The constants $\rho$ and $\alpha$ are as in Theorem \ref{lem:spatdec}.
\end{theorem}
Before proving the result above, we first revisit the KKT conditions of the optimization problems in \eqref{eq:cent_mpc} and \eqref{eq:dist_mpc} and then we introduce two key auxiliary results to proving Theorem \ref{thm:diffdec}.

Let $t \in [T]$ be arbitrary and $\ell \leq k$ be the horizon variable. Then we define the total cost as 
\begin{equation}
    \label{eq:totcost}
    \hat{f}(z) := \sum_{\tau=0}^{\ell-1}f_{t+\tau}(y_\tau) + \sum_{\tau=1}^{\ell}c_{t+\tau}(v_{\tau-1}) + g(y_{\ell}),
\end{equation}
where $g(\cdot)$ is either $F(\cdot)$ or $f_T(\cdot)$ and $z := (y_0. v_0, ..., y_{\ell-1},v_{\ell-1},y_{\ell})$ is the trajectory. 

Let $J$ be the constraint jacobian 
\begin{equation}\label{eq:jacobian}J(A,B) := \begin{bmatrix}
    I \\ -A & -B & I \\ & & \ddots \\ & & -A & -B & I
\end{bmatrix},\end{equation}
which has $\ell + 1$ rows. Let, $\lambda$ be the dual variables. The KKT conditions of \eqref{eq:dist_mpc} are then 
\begin{equation}\label{eq:nonlinkkt}\begin{bmatrix}
    \nabla \hat{f}(z^d) + (J^d)^\top \lambda^d \\ J^dz^d 
\end{bmatrix} = \begin{bmatrix}
    0 \\ \begin{bmatrix}
        x^{(i,\kappa)} \\ \zeta^{(i,\kappa)} 
    \end{bmatrix}
\end{bmatrix},\end{equation}
where $J^d := J(A^{(i,\kappa)},B^{(i,\kappa)})$ and $(z^d, \lambda^d)$ is primal-dual solution of $\dpc_t^\ell$. Further, we have the following lemma:
\begin{lemma}
[Lemma 1 in \cite{QuPDGD2019}]\label{lem:grad2mat}
    For $\mu$-strongly convex, $L$-smooth, and twice continuously-differentiable $\hat{f}$, For each $z$ and $z'$ there exists symmetric $G(z,z')$ such that $\mu I \preceq G(z) \preceq LI$ and $\nabla \hat{f}(z) - \nabla \hat{f}(z') = G(z,z')(z-z')$.
\end{lemma}

 Using Lemma \ref{lem:grad2mat}, we can write \eqref{eq:nonlinkkt} as $\tilde{H}^dq^d = b^d$ where
\begin{equation}\label{eq:kktmatdefs}\tilde{H}^d := \begin{bmatrix}
    G(z^d, 0) & (J^d)^\top \\ J^d
\end{bmatrix}, \quad b^d := \begin{bmatrix}
    0 \\ \begin{bmatrix}
        x^{(i,\kappa)} \\ \zeta^{(i,\kappa)}
    \end{bmatrix}
\end{bmatrix},\end{equation} and $q^d$ is short hand notation for $(z^d, \lambda^d)$. Furthermore, when taking the difference of the KKT conditions for equations \eqref{eq:cent_mpc} and \eqref{eq:dist_mpc}, we get, 
$$\begin{bmatrix}
    \nabla \hat{f}(z^c) - \nabla \hat{f}(z^d) + (J^c)^\top \lambda^c - (J^d)^\top \lambda^d \\ J^cz^c - J^d z^d  
\end{bmatrix} = \begin{bmatrix}
    0 \\ \begin{bmatrix}
        x^\perp \\ \zeta^\perp 
    \end{bmatrix}
\end{bmatrix},$$
where the notation $z^c$ and $J^c := J(A,B)$ denote the solution and constraint jacobian of $\pc_t^\ell$ respectively and the vectors on the RHS are $x^\perp := x-x^{(i,\kappa)}$ and $\zeta^\perp := \zeta-\zeta^{(i,\kappa)}$. Applying Lemma \ref{lem:grad2mat} to $\nabla \hat{f}(z^c) - \nabla \hat{f}(z^d)$, we get
\begin{equation}\label{eq:kktdiffs}
    \underbrace{\begin{bmatrix}
        G(z^c,z^d) & (J^c)^\top \\
        J^c
    \end{bmatrix}}_{:=H^c}q^c - \underbrace{\begin{bmatrix}
        G(z^c, z^d) & (J^d)^\top \\ J^d
    \end{bmatrix}}_{:=H^d}q^d = \underbrace{\begin{bmatrix}
        0 \\ \begin{bmatrix}
            x^\perp \\ \zeta^\perp
        \end{bmatrix}
    \end{bmatrix} }_{:=b^\perp} 
\end{equation}
where $q^c$ is short hand for $(z^c,\lambda^c)$). 

The purpose of rewriting the KKT conditions using the $H$ matrices defined above is because we have the following result from \cite{Shin2022sensitivity} and \cite{Sungho2023nearopt} on such $H$ matrices: 
\begin{lemma}[Theorem 3.6 in \cite{Shin2022sensitivity} and Theorems A.3 and A.4 in \cite{Sungho2023nearopt}]\label{lem:spatdec}
     Consider the following matrix $H$ 
    $$H = \begin{bmatrix}
        G & J(A,B)^\top \\
        J(A,B) & 0 
    \end{bmatrix}, $$ such that $G$ is a block diagonal matrix with its singular values bounded as $\mu I \preceq G \preceq L I$, and the system $(A,B)$ is $(L,\gamma)$-stabilizable. 
Further, let $\mathcal{G} = \{\mathcal{V},\mathcal{E}\}$ be any graph whose set of nodes partitions $H$ such that for any $i,j \in \mathcal{V}$, if $\dist(i,j) > 1$, then $H[i,j] = 0$. Then, the KKT matrix $H$ has singular values bounded as $\mu_H \leq \sigma(H) \leq L_H$ and its inverse $H^{-1}$ is spatially exponentially decaying with respect to the graph $\mathcal{G}$ 
    \begin{equation}\label{eq:spatialexpdecay}
        \lVert H^{-1}[i,j]\rVert \leq \alpha\rho^{\dist(i,j)},
    \end{equation}
    where the constants are defined as: {\small
    $$\begin{aligned}
        \mu_J &:= \frac{(1-\gamma)^2}{L^2(1-L^2)}, \quad L_H := 2L+1 \\
        \mu_H &:= \left(\frac{1}{\mu} + \left(1 + \frac{2L_H}{\mu} + \frac{L_H^2}{\mu^2}\right)\frac{L_H}{\mu_J}\right)^{-1},\\
        \rho &:= \left(\frac{L_H^2 - \mu_H^2}{L_H^2 + \mu_H^2}\right)^{\frac{1}{2}}, \quad \alpha := \frac{L_H}{\mu_H^2\rho}.
    \end{aligned}$$ }
\end{lemma} 

With the results above, we are ready to prove Theorem \ref{thm:diffdec}.  

\begin{proof}[Proof of Theorem \ref{thm:diffdec}]
    By \eqref{eq:kktdiffs}, we have,
    $$\begin{aligned}&b^\perp = H^cq^c - H^dq^d = H^c(q^c-q^d) - (H^d-H^c)q^d \\
    \implies &q^c-q^d = (H^{c})^{-1}\left(b^\perp - H^\perp q^d \right)\end{aligned}$$
    where $H^\perp =H^c - H^d$. Then we obtain the upper bound: {\footnotesize
    \begin{align} 
        &\lVert q^c[i] - q^d[i]\rVert = \bigg\lVert \sum_{j \in \mathcal{V}} (H^c)^{-1}[i,j](b^\perp[j] - \sum_{k \in \mathcal{V}}H^\perp[j,k]q^d[k]) \bigg \rVert \nonumber \\
        &\leq \alpha\sum_{j \in \mathcal{V}\setminus \knbhd[i]}\rho^{\dist(i,j)}\bigg(\lVert b^\perp \rVert \nonumber \\
        &+  \sum_{k \in \partial\knbhd[i]\cap\nbhd[j]}\sum_{m\in \knbhd[i]}\lVert H^\perp[j,k]\rVert \lVert (\tilde{H}^d)^{-1}[k,m]\rVert \lVert b^d[m]\rVert\bigg) \nonumber \\
        &\leq \alpha \sum_{j\in\mathcal{V}\setminus\knbhd[i]}\rho^{\dist(i,j)}\bigg(\lVert b^\perp \rVert \nonumber \\ &+ 2\alpha L\sum_{k \in \partial \knbhd[i]\cap \nbhd[j]}\sum_{m \in \knbhd[i]}\rho^{\dist(k,m)}\lVert b^d\rVert \bigg) \nonumber \\
        &\leq \alpha \sum_{d = \kappa + 1}^{\infty} (\rho/\delta_{\mathcal{S}} )^{d}p(d)\delta_{\mathcal{S}}^d \bigg(\lVert b^\perp \rVert \nonumber \\ &\qquad \qquad \qquad \qquad \qquad + 2\alpha L p(1)\sum_{s=0}^{2\kappa}(\rho/\delta_{\mathcal{S}})^{s}p(s)\delta_{\mathcal{S}}^{s}\lVert b^d\rVert \bigg)  \nonumber \\
        &\leq \frac{2\alpha^2\delta_{\mathcal{S}} L p(1)}{(1-\delta_{\mathcal{S}})^2}\left(\sup_{d \in \mathbb{Z}^+}\left(\frac{\rho}{\delta_{\mathcal{S}}}\right)^dp(d)\right)^2(\lVert b^\perp \rVert + \lVert b^d\rVert)\delta_{\mathcal{S}}^\kappa \nonumber 
    \end{align} }
    
   In the first inequality, we have used (i) Lemma \ref{lem:spatdec} applied to $H^{c}$; (ii) for $j  \in \knbhd[i]$, the $j$'th entry/row of $b^\perp$ and $H^\perp$ are zero (hence the outer sum is over $j\in \mathcal{V}\setminus \knbhd[i]$); (iii) $H^\perp[j,k]$ is nonzero only when $k\in \nbhd[j]$ since $H^\perp$  consists of networked matrices, and $q^d[k]$ is only nonzero in $\knbhd[i]$ (hence the inner sum is over $k \in \partial\knbhd[i]\cap\nbhd[j]$); (iv) $\tilde{H}^d q^d = b^d$. 
   The second inequality uses Lemma \ref{lem:spatdec} applied to $\tilde{H}^{d}$ and the fact that $\lVert H^\perp[j,k]\rVert \leq 2L$. 
   The third inequality  follows from Assumption \ref{assump:d_bands}.
   Finally, the inequality in \eqref{eq:diffdec} follows from $\lVert b \rVert \leq \lVert x \rVert + \sum_{\tau=0}^{\ell-1}\lVert \zeta_\tau \rVert$ for either $b^\perp$ or $b^d$. 
\end{proof} 

\subsection{Proof of Step 2 (ISS)}
To show ISS, we first show a recursive upper bound on the states generated by $DTPC_k$, denoted as $x_{0:T}$. 
\begin{theorem}\label{thm:statenorm}
    Let $C:= \max(\Gamma, \Omega)$ and we denote $\delta := \max(\delta_{\mathcal{S}}, \delta_{\mathcal{T}})$ where $\Gamma$ and $\delta_{\mathcal{S}}$ are as in Theorem \ref{thm:diffdec} and $\Omega$ and $\delta_{\mathcal{T}}$ are from Lemma \ref{lem:termlpz}. Under the Assumptions of Theorem \ref{thm:perf_guarantees}, we have the following upper bounds on the states generated by $DTPC_k$. For $1 \leq t +1\leq k$: {\small
    \begin{multline}\label{eq:snormk}
        \lVert x_{t+1}\rVert \leq C\sum_{m=0}^{t}(L_N\delta^{\kappa+m} + 2C^2\delta^{2k-m-3})\lVert x_{t-m}\rVert \\
        + C\lVert x_0 \rVert + W, 
    \end{multline} }
    for $k \leq t+1 \leq T-k$: {\small
    \begin{equation}\label{eq:snormTk}
        \lVert x_{t+1} \rVert \leq C\sum_{m=0}^{k-1}(L_N\delta^{\kappa+m} + 2C^2\delta^{2k-m-3})\lVert x_{t-m}\rVert + W, 
    \end{equation}}
    and finally for $t+1 > T-k$: {\small
    \begin{multline}\label{eq:snormT}
        \lVert x_{t+1}\rVert \leq  \sum_{m=0}^{t+k-T}2CL_N\lVert x_{t-m}\rVert \delta^{\kappa+m} \\
            + C\delta^{t+k-T+1}\lVert x_{T-k}\rVert + W.
    \end{multline}}
    where $L_N := CLN$ and $W$ is as in Theorem \ref{thm:perf_guarantees}.
\end{theorem}
Theorem~\ref{thm:statenorm} directly leads to the ISS in \eqref{eq:distISS} by induction. 

\begin{proof}[Proof of \eqref{eq:distISS}]
First, it is easy to verify that the ISS bound holds for $t=0$. For the induction step, we show the case $2k\leq t\leq T-k-1$, and the other cases are similar. In other words, we assume the ISS bound holds for all $t_0\leq t$ for some $ t\in[2k, T-k-1]$, and now show the ISS bound also holds for $t+1$. Since $k<t+1 \leq T-k$, we have from Theorem \ref{thm:statenorm} that {\small
\begin{equation}\label{eq:sub1} \lVert x_{t+1}\rVert \leq C\sum_{m=1}^{k}(L_N\delta^{\kappa+m-1} + 2C^2\delta^{2k-m-2})\lVert x_{t-m+1}\rVert + W.\end{equation}}
By the induction hypothesis, the upper bound in \eqref{eq:sub1} becomes,  
\begin{align}\label{eq:sub2} &C\sum_{m=1}^{k}\overbrace{(L_N\delta^{\kappa+m-1} + 2C^2\delta^{2k-m-2}) }^{:=\Lambda_m}\nonumber\\ \cdot &\left(\frac{C}{\xi}(1-\xi)^{t-m-k+1}\lVert x_0\rVert + \frac{W}{\xi}\right) + W,\end{align}
where the $\max$ disappears since $t \geq 2k-1$. Then, note the choice of $k$ and $\kappa$ in Theorem \ref{thm:perf_guarantees} guarantees that $C\sum_{m}\Lambda_m \leq 1-\xi$, 
which allows us to upper bound \eqref{eq:sub2} as

\begin{align}\label{eq:sub3} &\left(\frac{C}{\xi}(1-\xi)^{t-k+1}\lVert x_0\rVert \right)C\sum_{m=1}^{k}\Lambda_m(1-\xi)^{-m}+\frac{W}{\xi}.\end{align} Finally, the selection of $1-\xi = \sqrt{\delta}$, $k$, and $\kappa$ in Theorem \ref{thm:perf_guarantees} further guarantees that 
$C\sum_{m}\Lambda_m(1-\xi)^{-m} \leq 1,$
which obtains the desired bound in \eqref{eq:distISS} and concludes the induction.  \end{proof}  

The remainder of this subsection will be devoted to proving Theorem \ref{thm:statenorm}. We first require a couple definitions and supplementary results. First, we define the following terminal state predictive control problem similar to \eqref{eq:cent_mpc}: 
{\small
\begin{align}\label{eq:term_mpc}
    &\pct_t^k(x,\zeta,\bar{x}) := \\
    \argmin_{y_{0:k}, v_{0:k-1}} &\sum_{\tau=0}^{k}f_{t+\tau}(y_\tau) + \sum_{\tau=1}^{k}c_{t+\tau}(v_{\tau-1}), \nonumber \\
    \mathrm{s.t.} \: \: &y_{\tau+1} = Ay_\tau + Bv_\tau + \zeta_\tau, \quad \forall \tau \in [k-1], \nonumber \\
    & y_0 = x, \quad y_k = \bar{x} , \nonumber
\end{align} }
By \cite{Lin2021perturb}, the following lipschitz property of $\pct_t^k$ holds.  
\begin{lemma}\label{lem:termlpz}
    (Theorem 3.3 in \cite{Lin2021perturb}) Under assumptions \ref{assump:sys} and \ref{assump:costs} and for horizon length $k \geq d$, the controllability index, given any $(x,\zeta,\bar{x})$ and $(x', \zeta', \bar{x}')$, {\small 
    \begin{align}\label{eq:termlpsz}
        &\lVert \pct_t^{k}(x,\zeta,\bar{x})_{y_m} - \pct_t^{k}(x',\zeta',\bar{x}')_{y_m} \rVert  \nonumber \\
        &\leq \Omega\left(\delta_{\mathcal{T}}^{m}\lVert x-x'\rVert + \delta_{\mathcal{T}}^{k-m}\lVert \bar{x} - \bar{x}'\rVert + \sum_{l=0}^{k-1}\delta_{\mathcal{T}}^{\lvert m-l\rvert}\lVert \zeta_l - \zeta_l'\rVert \right) 
    \end{align} }
    where $D := \sup_{t \in [T-k]} \lVert w_t\rVert$ and the constants $\Omega$ and $\delta_{\mathcal{T}}$ are defined in Theorem 3.3 in \cite{Lin2021perturb}
\end{lemma}
Furthermore, as a Corollary of Lemma \ref{lem:spatdec}, we have the following lipschitz result on $\pc_t^\ell$ 
\begin{corollary}\label{cor:lpsz}
    Let $\ell \in [T]$ and $q$ denote the optimal vector of primal and dual variables to $\pc_t^\ell(x,\zeta; g(\cdot))$ and $q'$ be that of $\pc_t^\ell(x',\zeta'; g(\cdot))$. We partition $q$ and $q'$ by the the temporal graph $\mathcal{G}_{\ell} := \{ \{[\ell]\}, \{(0,1),..., (\ell-1,\ell)\}\}$ such that we can denote $q_m$ and $q_m'$ indexed by $m \in [\ell]$. The lipschitz property then follows from Lemma \ref{lem:spatdec}: 
    {\small
    \begin{align}\label{eq:lpsz}
        \lVert q_m-q_m'\rVert & \leq \alpha\left(\rho^{m}\lVert x-x'\rVert + \sum_{l=0}^{\ell-1}\rho^{\lvert m-l\rvert}\lVert \zeta_{l} - \zeta'_{l}\rVert \right)\\
        &\leq \alpha\left(\rho^{m}\lVert x-x'\rVert + \frac{2D}{1-\rho} \right),
    \end{align} }
\end{corollary}
Utilizing the above two results, we attain the following:
\begin{lemma}\label{lem:diffhorizon}
    Let $y_{m+1} := \pc_{t-m}^k(x_{t-m})_{y_{m+1}}$ and $y_{m+1}' := \pc_{t-m-1}^{k}(x_{t-m-1})_{y_{m+2}}$, then we have the following upper bound on the norm of their difference: {\small
    \begin{multline}\label{eq:diffhorizon}
        \lVert y_{m+1} - y_{m+1}'\rVert \leq CL_N(\lVert x_{t-m-1}\rVert + D_k)\delta^{\kappa+m+1} \\ + C^2\delta^{k-m-2}\left(\delta^{k-1}(\lVert x_{t-m}\rVert + C\delta\lVert x_{t-m-1}\rVert) + \frac{6C}{1-\delta}D_k\right),
    \end{multline} }
\end{lemma}
The proof is in \Cref{subsec:app_2}, 
from which we acquire a few auxiliary results. First, by the principle of optimality, we have \begin{equation}\label{eq:pop}\pc_{t-m-1}^k(x_{t-m-1})_{y_{m+2}} = \pc_{t-m}^{k-1}(x_{t-m}^c)_{y_{m+1}},
\end{equation}
where $x_{t-m}^c := \pc_{t-m-1}^k(x_{t-m-1})_{y_1}$. Next, we also have
\begin{equation}\label{eq:cntrlbnd}
    \lVert x_{t-m} - x_{t-m}^c\rVert \leq L_N(\lVert x_{t-m-1}\rVert + D_k)\delta^\kappa,
\end{equation}
which follows directly from rewriting the norm as
\begin{equation}\label{eq:ctbnd1}\lVert x_{t-m} - x_{t-m}^c\rVert = \lVert B(u_{t-m-1} - \pc_{t-m-1}^k(x_{t-m-1})_{v_0})\rVert, \end{equation}
where $u_{t-m-1} := (\dpc_{t-m-1}^{k}(x_{t-m-1}, \knbhd[i])_{v_0[i]})_{i \in \mathcal{V}}$, and then applying Theorem \ref{thm:diffdec} to get the upper bound in \eqref{eq:cntrlbnd}.
Now we may prove Theorem \ref{thm:statenorm}.
\begin{proof}[Proof of Theorem~\ref{thm:statenorm}]  We only prove the case for $k \leq t \leq T-k-1$, and the proofs for $t \leq k-1$ and $t \geq T-k$ are similar.
    We start by comparing the norm of $\lVert x_{t+1}\rVert$ as {\footnotesize
    $$\begin{aligned}&\lVert x_{t+1}\rVert = \left\lVert \pc_t^k(x_t)_{y_1} - (x_{t+1}^c - x_{t+1})\right\rVert \\
    &\leq \sum_{m = 0}^{k-2}\lVert \pc_{t-m}^k(x_{t-m})_{y_{m+1}}-\pc_{t-m-1}^k(x_{t-m-1})_{y_{m+2}}\rVert \\ &+ \lVert \pc_{t-k+1}^k(x_{t-k+1})_{y_k}\rVert + L_N(\lVert x_t\rVert + D_k)\delta^\kappa \\
    &\leq \sum_{m=0}^{k-2} \bigg(CL_N(\lVert x_{t-m-1}\rVert + D_k)\delta^{\kappa+m+1} \\ &+ C^2\delta^{k-m-2}\left(\delta^{k-1}\Big(\lVert x_{t-m}\rVert + C\delta\lVert x_{t-m-1}\rVert\Big) + \frac{6C}{1-\delta}D_k\right) \bigg)\\
    &+ C\delta^{k}\lVert x_{t-k+1}\rVert  + \frac{2C}{1-\delta}D_k + L_N(\lVert x_t\rVert + D_k)\delta^\kappa \\
    &\leq C\sum_{m=0}^{k-1}\left(L_N\delta^{\kappa+m} + 2C^2\delta^{2k-m-3}\right)\lVert x_{t-m}\rVert  + W.
    \end{aligned}$$ }The first equality uses the definition $ x_{t+1}^c = \pc_t^k(x_t)_{y_1}$. In the first inequality, we telescope $\pc_t^k(x_t)_{y_1}$ and apply \eqref{eq:cntrlbnd}. In the second inequality, we apply Lemma \ref{lem:diffhorizon} and Corollary \ref{cor:lpsz} with $x' = 0$ and $\zeta' = 0$. In the final inequality, we combine terms and upper bound by the geometric sum to attain the desired bound. 
\end{proof}

\subsection{Proof of Step 3 (Regret)} 
To prove the regret result in Theorem \ref{thm:perf_guarantees} we first prove the following critical result which bounds the difference between $DTPC_k$'s trajectory and the optimal offline trajectory: 
\begin{theorem}
    \label{thm:statediffbnd}
    Let $x_{0:T}$ be the trajectory generated by $DTPC_k$ and $x^*_{0:T}$ is the optimal offline trajectory. For $t+1 \leq T-k$, we have that {\small
    \begin{align}\label{eq:statediffbndTk} &\lVert x_{t+1} - x_{t+1}^*\rVert \leq C\delta^{k}\left(2C\delta^T\lVert x_0\rVert + \frac{4C}{1-\delta}D_k\right) +\nonumber \\&\sum_{m=0}^{t}\frac{4C^2L_N}{\delta(1-\delta)^2}\delta^{m}\left(\left(\delta^\kappa + \delta^{2k}\right)\lVert x_{t-m}\rVert + (\delta^\kappa + \delta^k)D_k\right), \end{align} }
    and for $t+1 \geq T-k+1$, we have that {\small 
    \begin{align}\label{eq:statediffbndT}
        &\lVert x_{t+1} - x_{t+1}^* \rVert \leq \sum_{m=0}^{t}CL_N(\lVert x_{t-m}\rVert+D_k)\delta^{\kappa+m} + \nonumber \\ & \sum_{m=1}^{T-k}\frac{4C^3}{\delta(1-\delta)^2}\delta^{t+k-T+m}\left(\delta^{2k}\lVert x_{T-k-m}\rVert + \delta^kD_k\right),
    \end{align} }
\end{theorem}
 \begin{proof} 
 \textit{For $t+1 \leq T-k$:} Let $\hat{x}_{0:T}$ denote the state trajectory of $\pc_0^T(x_0;F)$. Then we can upper bound the norm of the difference as
$$\lVert x_{t+1} - x_{t+1}^*\rVert \leq \lVert x_{t+1} - \hat{x}_{t+1}\rVert + \lVert \hat{x}_{t+1} - x_{t+1}^*\rVert.$$ By Lemma \ref{lem:termlpz} and Corollary \ref{cor:lpsz} and for $t+1 \leq T-k$, we have the upper bound from equation (21) in \cite{Lin2021perturb}
\begin{align}\label{eq:compbnd} &\lVert\hat{x}_{t+1}-x_{t+1}^*\rVert \leq C\delta^k\left(2C\delta^T\lVert x_0\rVert + \frac{4C}{1-\delta}D_k\right).\end{align}
Then we make the observation that 
\begin{equation}\label{eq:obs_pop}\pc_{t-m-1}^{T-t+m+1}(x_{t-m-1})_{y_{m+2}} = \pc_{t-m}^{T-t+m}(x_{t-m}^{c'})_{y_{m+1}},\end{equation}
where $x_{t-m}^{c'} := \pc_{t-m-1}^{T-t+m+1}(x_{t-m-1})_{y_1}$. Thus we have,
{\footnotesize
$$\begin{aligned}
    &\lVert x_{t+1} - \hat{x}_{t+1}\rVert = \lVert x_{t+1} - \pc_0^T(x_0;F)_{y_{t+1}}\rVert \\ 
&\leq \lVert x_{t+1} - \pc_{t}^{T-t}(x_t)_{y_1}\rVert \\ &+ \sum_{m=0}^{t-1}\lVert \pc_{t-m}^{T-t+m}(x_{t-m})_{y_{m+1}} - \pc_{t-m-1}^{T-t+m+1}(x_{t-m-1})_{y_{m+2}} \rVert \\ 
&\leq \lVert x_{t+1} - x_{t+1}^{c'}\rVert + \sum_{m=0}^{t-1}C\delta^{m+1}\lVert x_{t-m} - x_{t-m}^{c'}\rVert  \\
&\leq \lVert x_{t+1} - x_{t+1}^c\rVert + \lVert x_{t+1}^c - x_{t+1}^{c'}\rVert \\ &+ 
\sum_{m=0}^{t-1}C\delta^{m+1}\left(\lVert x_{t-m}-x_{t-m}^{c}\rVert + \lVert x_{t-m}^c - x_{t-m}^{c'}\rVert\right), 
\end{aligned}$$ }where we've applied the observation from above \eqref{eq:obs_pop} and Corollary \ref{cor:lpsz} in the second inequality.
 We then apply the bound in \eqref{eq:cntrlbnd} and the bound
\begin{equation}\label{eq:Tminustbnd}\lVert \pc_t^k(x_t)_{y_1} - \pc_t^{T-t}(x_t)_{y_1}\rVert \leq \frac{4C^2}{\delta(1-\delta)^2}(\delta^{2k}\lVert x_t\rVert + \delta^kD_k),\end{equation}
(Equation 19 in \cite{Lin2021perturb}) which holds for $t +1 \leq T-k$. This completes the proof for $t+1 \leq T-k$. The case for $t+1 \geq T-k+1$ is similar so its proof is omitted here. 
\end{proof} 

Before finishing the proof of Theorem \ref{thm:perf_guarantees}, we require the following two Lemmas. First, we have from \cite{Lin2021perturb}:
\begin{lemma}[Lemma F.2 in \cite{Lin2021perturb}]\label{lem:etacost}
    Suppose $h:\mathbb{R}^{n} \mapsto\mathbb{R}_+$ is a convex and $L$-smooth continuously differentiable function. Then for $x,x'\in\mathbb{R}^{n}$ and $\eta > 0$, we have that
    $$h(x) - (1+\eta)h(x') \leq \frac{L}{2}\left(1 + \frac{1}{\eta}\right)\lVert x-x'\rVert^2.$$
\end{lemma}
\noindent The above Lemma \ref{lem:etacost} will be applied to both $f_t$ and $c_t$. 

In addition, we require the following bound on the one step terminal state problem ($k = 1$) in \eqref{eq:term_mpc}, the proof is in \Cref{subsec:app_3}.
\begin{lemma}\label{lem:onestepterm}
    Let $v$ and $v'$ be the two ``one-step'' terminal state problems: $\pct_t^1(x_t,x_{t+1})_{v_0}$ and $\pct_t^1(x_t',x_{t+1}')_{v_0}$. Under Assumptions \ref{assump:sys} and \ref{assump:costs}, we have that 
    \begin{equation}\label{eq:onestepterm}
        \lVert v-v'\rVert^2 \leq C^2\left(\lVert x_t-x_t'\rVert^2  + \lVert x_{t+1}-x_{t+1}'\rVert^2\right), 
    \end{equation}
    where $C$ is as in Theorem \ref{thm:statenorm}.
\end{lemma}
We are now ready to prove the regret in Theorem \ref{thm:perf_guarantees}.
\begin{proof}[Proof of Theorem \ref{thm:perf_guarantees}]
    First, denote $\bar{u}_t := \pct_t^1(x_t,x_{t+1})_{v_0}$. Then consider the difference of the costs $c_{t+1}$ which can be written as: 
    {{\footnotesize
     \begin{align}\label{eq:controlcostbound} &c_{t+1}(u_t) - (1+\eta)c_{t+1}(\bar{u}_t) + (1+\eta)\left(c_{t+1}(\bar{u}_t) - (1+\eta)c_{t+1}(u_t^*)\right) \nonumber \\
    &\leq \frac{L}{2}\left(1+\frac{1}{\eta}\right)\left(\lVert u_t - \bar{u}_t\rVert^2 +(1+\eta)\lVert \bar{u}_t - u^*_t)\rVert^2 \right) \nonumber \\
    &\leq C^3\left(1+\frac{1}{\eta}\right)\bigg(L_N^2(\lVert x_t\rVert + D_k)^2\delta^{2\kappa} + \lVert x_{t+1}-x_{t+1}^c\rVert^2  \nonumber 
 \\ &+ (1+\eta)\left(\lVert x_{t}-x_{t}^*\rVert^2  + \lVert x_{t+1}-x_{t+1}^*\rVert^2\right)\bigg) \nonumber \\
 &\leq 2C^3L_N^2\left(1+\frac{1}{\eta}\right)\bigg((\lVert x_{t} \rVert + D_k)^2\delta^{2\kappa} +  (1+\eta)(\lVert x_t-x_t^*\rVert^2 \nonumber  \\ &+ \lVert x_{t+1}-x_{t+1}^*\rVert^2)\bigg).\end{align} }}
 
 Where in the first inequality, we have applied Lemma \ref{lem:etacost}. To obtain the second inequality, by the parallelogram identity, we have the upper bound
 \begin{equation}\label{eq:split1}\lVert u_t - \bar{u}_t\rVert^2  \leq 2\left(\lVert u_t - \pc_t^{k'}(x_t)_{v_0}\rVert^2 + \lVert \pc_t^{k'}(x_t)_{v_0} - \bar{u}_t\rVert^2\right),\end{equation} 
 where $k' := \min(k, T-t)$. Applying Theorem \ref{thm:diffdec} to $\lVert u_t-\pc_t^{k'}(x_t)_{v_0}\rVert$ and Lemma \ref{lem:onestepterm} to $\lVert \pc_t^{k'}(x_t)_{v_0} - \bar{u}_t\rVert$ gives us the following upper bound of \eqref{eq:split1}: 
 \begin{align}
     2\left(L_N^2(\lVert x_t\rVert + D_k)^2\delta^{2\kappa} + C^2(\lVert x_{t+1} - x_{t+1}^c\rVert^2)\right). \label{eq:split2}
 \end{align}
Observe that
 $$\lVert \bar{u}_t - u_t^*\rVert = \lVert \pct_t^1(x_t,x_{t+1})_{v_0} - \pct_t^1(x_t^*, x_{t+1}^*)_{v_0}\rVert,$$ for which we can apply Lemma \ref{lem:onestepterm}. This, together with \eqref{eq:split2} gives us the second inequality in \eqref{eq:controlcostbound}. 
 Finally, we attain the last inequality by applying \eqref{eq:cntrlbnd} again. We take $\eta = \Theta(\max(\delta^k,\delta^\kappa))$. 
 Denoting $1+\eta' = (1+\eta)^2$, we have {\footnotesize
 \begin{align}\label{eq:etaregret}
     &\cost(DTPC_k) - (1+\eta')\cost(OPT) \nonumber \\
     &= \sum_{t=1}^{T}(f_{t}(x_{t}) - (1+\eta')f_{t}(x_{t}^*)) + (c_{t}(u_{t-1}) - (1+\eta')c_{t}(u_{t-1}^*)) \nonumber\\ 
    &\leq 2C^3L_N^2\left(1+\frac{1}{\eta}\right)\sum_{t=0}^{T-1}\bigg( \lVert x_{t+1}-x_{t+1}^*\rVert^2 +(\lVert x_t\rVert  + D_k)^2\delta^{2\kappa} \nonumber \\ &+ (1+\eta)\left(\lVert x_t-x_t^*\rVert^2 + \lVert x_{t+1} - x_{t+1}^*\rVert^2\right)\bigg) \nonumber \\
    &\leq 6C^3L_N^2\left(1+\frac{1}{\eta}\right)(1+\eta)\sum_{t=0}^{T}((\lVert x_t\rVert + D_k)^2\delta^{2\kappa} + \lVert x_t - x_t^*\rVert^2) \nonumber \\
    & = \frac{1}{\eta}O\bigg(\bigg[\left(D_k + \frac{\lVert x_0 \rVert + D_k}{\xi^2}\right)^2\delta^{2\kappa} \nonumber \\ &+  \left(D_k + \frac{\delta^{k}(\lVert x_0\rVert + D_k)}{\xi}\right)^2\delta^{2k}\bigg]T\bigg).
 \end{align} }
 
 Where in the first inequality we apply \eqref{eq:controlcostbound}, Lemma \ref{lem:etacost}, and that $\eta \leq \eta'$. In the second inequality, we merge all of the norms of differences. In the final line, we apply the ISS bound and Theorem \ref{thm:statediffbnd} to attain the final expression in \eqref{eq:etaregret} 
 
 Since $\eta' = 2\eta + \eta^2$, and hence, $\eta' = \Theta\left(\max(\delta^k,\delta^\kappa)\right)$, we simply add $\eta'\cost(OPT)$ to the RHS of \eqref{eq:etaregret} to achieve the bound in \eqref{eq:distREG}. This finishes the proof of Theorem \ref{thm:perf_guarantees}. 
 \end{proof}

\section{Simulations}
    We consider the temperature control of a building with HVAC network graph that is a $5\times 5$ mesh grid with diameter $8$ and each of its $25$ nodes corresponding to a different zone in the building. The states are the temperatures of each zone and their respective temperature integrators and the control variable at zone $i$ is its manipulated heat generation/absorption. Further details about the system and setup are in \Cref{subsec:app_sim}.
\begin{figure}
        \centering
        \includegraphics[height = 4cm, width = 8cm]{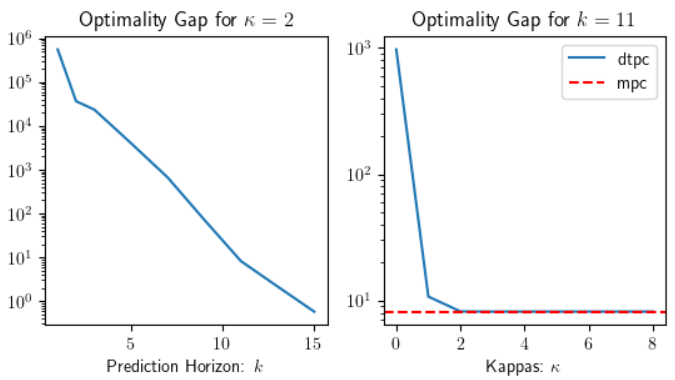}
        \caption{$DTPC_k$ versus the $OPT$ for fixed $\kappa$ (left) and fixed $k$ (right). }
        \label{fig:mesh}
    \end{figure}

\begin{figure}
    \centering
    \includegraphics[height = 4cm, width = 6cm]{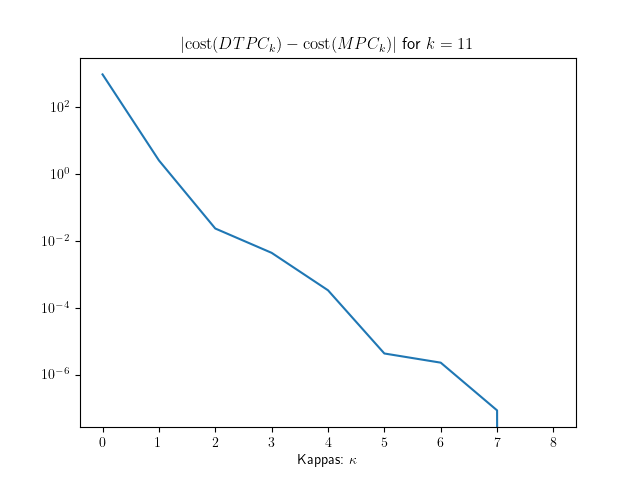}
    \caption{$DTPC_k$ versus $PC_k$ for $k = 11$ and varying $\kappa$.}
    \label{fig:DTPC_MPC}
\end{figure}

 The results are shown in Figure \ref{fig:mesh} where we chose a time horizon of $T = 30$. Observe that $DTPC_k$ exhibits the decaying regret behavior as in \eqref{eq:distREG} as $\kappa$ and $k$ increase. 
 Note that in Figure \ref{fig:mesh} (right), the regret stops decreasing because the prediction horizon $k$ becomes the bottleneck after $\kappa$ reaches $2$. Figure \ref{fig:DTPC_MPC} demonstrates our result in Theorem \ref{thm:diffdec}: when $\kappa$ increases under fixed $k$, $DTPC_k$ becomes exponentially close to $PC_k$.

 \noindent \textbf{Uncertainty Simulation}: 
 
 \begin{figure}[h]
    \centering
    \includegraphics[height = 5cm, width = 7cm]{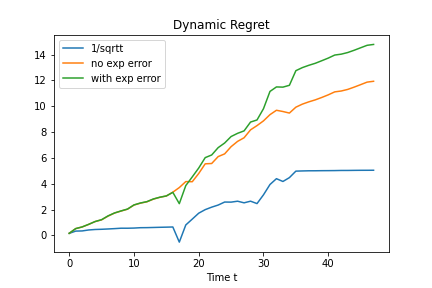}
    \caption{$uDTPC_k$ with different prediction errors for $k = 10$ (5 hours) and $\kappa = 3$.}
    \label{fig:uDTPC}
\end{figure}

 Here, we consider the same $5\times 5$ building system as before, but with forecasted PV generation and external temperatures. The details of the simulation are deferred to Appendix-E. 

 The results of our simulation are shown in Figure \ref{fig:uDTPC} with time horizon $T = 48$. In our simulation, we consider three different prediction errors. The first two prediction errors we consider are the aforementioned examples, $\phi^e$ and $\phi^c$ from equations \eqref{eq:sqrtt} and \eqref{eq:constexpdeg} which correspond to the dynamic regrets in blue and green, and the last prediction error we consider is constant such that
 \begin{equation}\label{eq:consterr}
     \phi_{t+n|t}^{const} \leq R,
 \end{equation}
 for some $R \geq 0$, which corresponds to the dynamic regret in orange. As expected, the growth of the dynamic regret for each example is bounded according to \eqref{eq:uncert_reg_bnd}. Worth noting is how little difference there is between the dynamic regrets obtained from performing $uDTPC_k$ under the predictions $\phi^c$ and $\phi^{const}$. The small difference between the two reflects how the prediction error term in \eqref{eq:uncert_reg_bnd} decays exponentially with look-ahead predictions further than the current step; in other words, good and accurate myopic predictions will furnish better performance than "decent" predictions over the whole prediction horizon.

\section{Conclusion}
In this work we have shown that our algorithm $DTPC_k$ produces trajectories that are similar to those produced by centralized predictive control $PC_k$. Furthermore, we extend $DTPC_k$ to the uncertain (oracle) predictions case and develop the $uDTPC_k$ algorithm for which we have shown stability and regret bounds for both $DTPC_k$ and $uDTPC_k$ algorithms which guarantee near-optimal performance. Additionally, we have shown that the performance of $uDTPC_k$ is heavily linked to the accuracy of the estimator; in particular, myopic predictions (those close to the current time step) exponentially matter more than those further in the future. In terms of future work, we would like to extend this analysis to the LTV dynamics case.

\bibliographystyle{ieeetr}
\bibliography{refs}

\fvtest{

\appendix
\subsection{Proofs of intermediary results in Step 1}\label{subsec:app_1}
\subsubsection{Proof of Lemma \ref{lem:spatdec}}
We note that the proof is a combination of Theorem 3.6 in \cite{Shin2022sensitivity} and Theorems A.3 and A.4 in \cite{Sungho2023nearopt}. For completeness, we provide a proof below with a focus on the differences from the proofs of those Theorems; the steps that are identical are omitted. 

The proof for the exponential decay result in equation \eqref{eq:spatialexpdecay} is available in \cite{Shin2022sensitivity} (Theorem 3.6)  which requires upper and lower bounds on the singular values of the $H$ matrix: i.e., $\mu_H \leq \sigma(H) \leq L_H$. 
So we will prove the bounds on the singular values similar to Theorems A.3 and A.4 in \cite{Sungho2023nearopt}. 

 Let $Z$ be an orthonormal null-space basis matrix for $J$, i.e. $\mathrm{Col}(Z) = \mathrm{Null}(J)$, then, we shall prove the following uniform regularity conditions:
\begin{equation}\label{eq:unifreg}
    H \preceq L_HI, \quad Z^\top G Z \succeq \mu I, \quad JJ^\top \succeq \mu_J I,
\end{equation}

\textit{Proof of $H \preceq L_H I$:} Let $H_{\tau,\tau'}$ denote the submatrix of $H$ such that the rows correspond to $(y_\tau,v_\tau,\lambda_\tau)$ and the columns correspond to $(y_{\tau'}, v_{\tau'}, \lambda_{\tau'})$, then we have that 
\begin{equation}\label{eq:H_submats}H_{\tau,\tau'} = \begin{cases}
    \begin{bmatrix}
        G_{y_\tau} & & I \\
        & G_{v_{\tau}}  \\
        I
    \end{bmatrix}, & \tau=\tau' \neq \ell, \\
    \begin{bmatrix}
        & & 0 \\
        & & 0 \\
        -A & -B
    \end{bmatrix}, & \tau = \tau'+1 \neq \ell \\
    \begin{bmatrix}
        G_{y_\ell} & I \\ I
    \end{bmatrix}, & \tau = \tau' = \ell \\
    \begin{bmatrix}
        & & 0 \\
        -A & -B
    \end{bmatrix}, & \tau = \tau'+1 = \ell \\
    0, & \tau-\tau' > 1
\end{cases},\end{equation} 
where each $G_{(\cdot)}$ comes from defining the block diagonal matrices of $G$ such that 
$$G := \begin{bmatrix}
    G_{y_0} \\
    & G_{v_0} \\
    & & \ddots \\
    & & & G_{v_{\ell-1}} \\
    & & & & G_{y_{\ell}}
\end{bmatrix}$$
where each $G_{y_\tau}$ and $G_{v_\tau}$ corresponds to the predictive state and control $y_\tau$ and $v_\tau$ respectively. 
The upper bound $\lVert H \rVert \leq 2L + 1$ follows from the following result (Lemma 5.11 from \cite{Shin2022sensitivity}):
\begin{equation}\label{eq:normsubmat}
    \lVert M \rVert \leq \left(\max_{i \in \mathcal{V}}\sum_{j\in\mathcal{V}}\lVert M_{ij}\rVert\right)^{1/2}\left(\max_{j \in \mathcal{V}}\sum_{i\in\mathcal{V}}\lVert M_{ij}\rVert\right)^{1/2}.
\end{equation}
\textit{Proof of $Z^\top GZ \succ \mu I$}: The bound is immediate from $Z$ being full rank and $G \succ \mu I$. 

\textit{Proof of $JJ^\top \succ \mu_J I$:} The proof is the exact same as in Theorem A.4 of \cite{Sungho2023nearopt} so it is omitted here.

Finally, we can use the conditions in \eqref{eq:unifreg} to prove the lower bound of $H$. First we can rewrite the inverse of $H$ by its schur complement {\small
\begin{align}\label{eq:hinv}
    &H^{-1} = \begin{bmatrix}
        G & J^\top  \\ J 
    \end{bmatrix}^{-1} \nonumber\\
    &= \begin{bmatrix}
        G^{-1} - G^{-1}J^{\top}(G_J)^{-1}JG^{-1} & G^{-1}J^\top(G_J)^{-1} \\ 
      (G_J)^{-1}JG^{-1} & -(G_J)^{-1}
    \end{bmatrix},
\end{align} }

where $G_J := JG^{-1}J^{\top}$. So we may upper bound the norm of $H^{-1}$ as the sum of the norms of its sub-blocks as such: 
\begin{equation}\label{eq:hinvnorm}\lVert H^{-1}\rVert \leq \lVert G^{-1}\rVert + (1 + 2\lVert JG^{-1}\rVert + \lVert JG^{-1}\rVert^2)\lVert G_J^{-1}\rVert,\end{equation}
where $\lVert G^{-1}\rVert \leq \frac{1}{\mu}$, $\lVert JG^{-1}\rVert \leq \frac{L_H}{\mu}$, and we can obtain the upper bound on $G_J^{-1}$ as 
\begin{align}\label{eq:schurnorm}
    &\lVert (JG^{-1}J^\top)^{-1}\rVert = \lVert (JG^{-1/2}G^{-1/2}J^\top)^{-1}\rVert \nonumber  \\
       &= \lVert (G^{-1/2}(J^\top J)^{1/2}(J^\top J)^{1/2}G^{-1/2})^{-1}\rVert \nonumber \\
       &= \lVert G^{1/2}(J^\top J)^{-1}G^{1/2}\rVert \leq \lVert G \rVert \lVert (J^\top J)^{-1}\rVert \leq \frac{L_H}{\mu_J}.
\end{align}
Applying the upper bounds to \eqref{eq:hinvnorm}, we obtain
$$\lVert H^{-1} \rVert \leq \frac{1}{\mu} + \left(1 + \frac{2L_H}{\mu} + \frac{L_H}{\mu^2}\right)\frac{L_H}{\mu_J},$$
as desired. This completes the proof of Lemma 2. \qed

\subsection{Proofs of intermediary results in Step 2}\label{subsec:app_2}

\subsubsection{Complete proof of ISS in Theorem \ref{thm:perf_guarantees}}For the case of $t \leq k-1$, the induction step from \eqref{eq:snormk} is
$$\lVert x_{t+1}\rVert  \leq C\sum_{m=0}^{t}\Lambda_m \left(\frac{C}{\xi}\lVert x_0\rVert + \frac{W}{\xi}\right) + C\lVert x_0\rVert + W,$$
for which the choice of $k$ and $\kappa$ in Theorem \ref{thm:perf_guarantees} guarantees that $C\sum_{m=0}^{t-1}\Lambda_m \leq (1-\xi)$ and so
\begin{equation}\label{eq:tleqk}\lVert x_{t+1}\rVert \leq \frac{C}{\xi}\lVert x_0\rVert + \frac{W}{\xi},\end{equation}
as desired. 

\textit{For the case of $k \leq t < 2k-1$}, let $(t-m-k+1)_{+}$ denote $\max(0,t-m-k+1)$, then by Theorem \ref{thm:statenorm}, { \small 
$$\lVert x_{t+1} \rVert \leq C\sum_{m=1}^{k}(L_N\delta^{\kappa+m-1} + 2C^2\delta^{2k-m-2})\lVert x_{t-m+1}\rVert + W.$$}
And thus, by the induction hypothesis and the condition that $C\sum_{m}\Lambda_m \leq 1-\xi$, 
\begin{equation}\label{eq:snorm2k}\lVert x_{t+1}\rVert \leq C\sum_{m=1}^{k}\Lambda_m\left(\frac{C}{\xi}(1-\xi)^{(t-m-k+1)_{+}}\lVert x_0\rVert \right) + \frac{W}{\xi},\end{equation}
for which we can split the sum up as 
\begin{align}\label{eq:split2k}
    &\left(\frac{C}{\xi}(1-\xi)^{t-k+1}\lVert x_0\rVert \right)C\bigg(\sum_{m=1}^{t-k}\left(\Lambda_m(1-\xi)^{-m}\right) \nonumber \\
    &+\sum_{m=t-k+1}^{k}\Lambda_m (1-\xi)^{-(t-k+1)} \bigg) + \frac{W}{\xi}.
\end{align}
Equation \eqref{eq:split2k} is clearly upperbounded by
\begin{equation}\label{eq:splituppbnd}
    \left(\frac{C}{\xi}(1-\xi)^{t-k+1}\lVert x_0\rVert\right)C\sum_{m=1}^{k}\Lambda_m(1-\xi)^{-m} + \frac{W}{\xi},
\end{equation}
which is the same as in equation \eqref{eq:sub3}, thus, by the same logic, the desired bound in \eqref{eq:distISS} is obtained for this case.

\textit{For the case of $t > T-k$}, we have from \eqref{eq:snormT} that 
{\footnotesize
$$\begin{aligned} &\lVert x_{t+1}\rVert \leq \sum_{m=0}^{t+k-T}2CL_N\lVert x_{t-m}\rVert \delta^{\kappa+m} + C\delta^{t+k-T+1}\lVert x_{T-k}\rVert + W\\ &\leq \left(\frac{C}{\xi}\delta^{t+k-T+1}\lVert x_{T-k}\rVert + \frac{W}{\xi}\right) \left(\sum_{m=0}^{t+k-T}2CL_N\delta^{\kappa+m} + \xi\right)\\ &\leq \frac{C}{\xi}\delta^{t+k-T+1}\lVert x_{T-k}\rVert + \frac{W}{\xi} \\
 &\leq \frac{C}{\xi^2}\delta^{t+k-T+1}\left( C(1-\xi)^{T-2k}\lVert x_0\rVert + W\right) + \frac{W}{\xi},\end{aligned} $$}where the second inequality is obtained through the following fact for $t-m \geq T-k$ 
$$\lVert x_{t-m}\rVert \leq \frac{C}{\xi}\delta^{t+k-T+1}\lVert x_{T-k}\rVert + \frac{W}{\xi},$$ which can be easily proved by induction, similar to the previous induction argument for $t\leq k-1$, cf. eq. \eqref{eq:tleqk}; the third line comes from applying the condition $\sum_{m}2CL_N\delta^{\kappa+m} \leq (1-\xi)$ (guaranteed by the choice of $\kappa$ in Theorem \ref{thm:perf_guarantees}).  The final inequality is obtained by applying the ISS bound for $t = T-k$ in \eqref{eq:distISS}. This completes the proof of \eqref{eq:distISS} \qed

\subsubsection{Proof of Corollary \ref{cor:lpsz}}
As shown in Equation \eqref{eq:H_submats}, $H_{t,t'} = 0$ for $\lvert t-t'\rvert > 1$ meaning that $H$ satisfies the conditions of Lemma \ref{lem:spatdec} for the graph $\mathcal{G}_{\ell}$. Thus, we have that 
\begin{align}
    &\lVert q_m - q_m'\rVert  = \left\lVert \sum_{l \in [\ell]}H^{-1}_{m,l}(b_{l}-b'_{l}) \right\rVert \nonumber \\  &\leq \alpha\left(\rho^{m}\lVert x-x'\rVert + \sum_{l=0}^{\ell}\rho^{\lvert m-l\rvert}\lVert \zeta_l - \zeta_l'\rVert  \right),
\end{align}
where we denote 
$$b_l := \begin{cases}
    \begin{bmatrix}
    0 \\ x
\end{bmatrix}, & l = 0, \\
\begin{bmatrix}
    0 \\ \zeta_l
\end{bmatrix}, & \mathrm{else},
\end{cases}$$
and similarly for $b_l'$. The inequality is obtained through application of Lemma \ref{lem:spatdec}.

\subsubsection{Proof of Lemma \ref{lem:diffhorizon}} First, observe that 
\begin{equation}\label{eq:obs_pop1}y_{m+1}' = \pc_{t-m-1}^{k}(x_{t-m-1})_{y_{m+2}} = \pc_{t-m}^{k-1}(x_{t-m}^c)_{y_{m+1}},\end{equation}
where $x_{t-m}^c := \pc_{t-m-1}^{k}(x_{t-m-1})_{y_1}.$ Then, denote 
$$\bar{y} := \pc_{t-m}^{k}(x_{t-m})_{y_{k-1}}, \quad \bar{y}' := \pc_{t-m}^{k-1}(x_{t-m}^c)_{y_{k-1}},$$
so now we may proceed to prove the Lemma as: {\footnotesize
\begin{align}
    &\lVert y_{m+1} - y_{m+1}'\rVert \nonumber \\ &= \lVert \pct_{t-m}^{k-1}(x_{t-m},\bar{y})_{y_{m+1}} - \pct_{t-m}^{k-1}(x_{t-m}^c, \bar{y}')_{y_{m+1}}\rVert \nonumber \\
    &\leq C\left(\delta^{m+1}\lVert x_{t-m} - x_{t-m}^c\rVert + \delta^{k-m-2}\lVert \bar{y} - \bar{y}'\rVert \right) \nonumber \\
    &\leq CL_N(\lVert x_{t-m-1}\rVert  + D_k)\delta^{\kappa+m+1} \nonumber\\ &+ C^2\delta^{k-m-2}\left(\delta^{k-1}\left(\lVert x_{t-m}\rVert + \lVert x_{t-m}^c\rVert \right) + \frac{4D_k}{1-\delta}\right) \nonumber \\
    &\leq  CL_N(\lVert x_{t-m-1}\rVert +D_k) \delta^{\kappa+m+1} \nonumber \\
    &+ C^2\delta^{k-m-2}\left(\delta^{k-1}\left(\lVert x_{t-m}\rVert + C\delta\lVert x_{t-m-1}\rVert\right) + \frac{6CD_k}{1-\delta}\right)
\end{align} }
where in the first equality, we apply the principle of optimality. In the first inequality we apply Lemma \ref{lem:termlpz}. In the second inequality, we apply the result in \eqref{eq:cntrlbnd} along with the upper bounds 
$$\lVert \bar{y}\rVert \leq C\left(\delta^{k-1}\lVert x_{t-m}\rVert + \frac{2D_k}{1-\delta}\right),$$
and
$$\lVert \bar{y}'\rVert \leq C\left(\delta^{k-1}\lVert x_{t-m}^c\rVert + \frac{2D_k}{1-\delta}\right),$$
which both come from the Lipschitz bound in Corollary \ref{cor:lpsz}. The last inequality then comes from applying Corollary \ref{cor:lpsz} to $x_{t-m}^c$ and combining the $D_k$ terms. \qed 

 \subsubsection{Full proof of Theorem \ref{thm:statenorm}}
 For the case $t \leq k-1$, we have that {\footnotesize
 \begin{align}
     &\lVert x_{t+1}\rVert = \lVert \pc_t^k(x_t)_{y_1} - (x_{t+1}^c - x_{t+1})\rVert \nonumber \\
     &\leq \sum_{m=0}^{t-1} \lVert \pc_{t-m}^k(x_{t-m})_{y_{m+1}}-\pc_{t-m-1}^k(x_{t-m-1})_{y_{m+2}}\rVert \nonumber \\ &+ \lVert \pc_0^k(x_0)_{y_{t+1}}\rVert + L_N(\lVert x_t\rVert + D_k)\delta^\kappa
    \nonumber \\
    &\leq \sum_{m=0}^{t-1} \bigg(CL_N\delta^{\kappa+m+1}(\lVert x_{t-m-1}\rVert + D_k) \nonumber \\ &+ C^2\delta^{k-m-2}\left(\delta^{k-1}\left(\lVert x_{t-m}\rVert + C\delta\lVert x_{t-m-1}\rVert\right) + \frac{6C}{1-\delta}D_k\right) \bigg) \nonumber \\
    &+ C\delta^{t+1}\lVert x_0\rVert + \frac{2C}{1-\delta}D_k + L_N(\lVert x_t\rVert + D_k)\delta^\kappa \nonumber \\
    &\leq \sum_{m=0}^{t}(CL_N\delta^{\kappa+m} + 2C^3\delta^{2k-m-3})\lVert x_{t-m}\rVert  + C\lVert x_0\rVert + W.
 \end{align} }
where we use the definition of $\pc_t^k(x_t)_{y_1}$ in the first equality. We telescope $\pc_t^k(x_t)_{y-1}$ and use the result in \eqref{eq:cntrlbnd} in the second line. In the second inequality, we apply Lemma \ref{lem:diffhorizon} and Corollary \ref{cor:lpsz} with $x' = 0$ and $\zeta' = 0$. Then in the last inequality we combine all the terms.

\textit{For $t \geq T-k$:} We continue similarly, 
{\footnotesize\begin{align}\label{eq:geqTminusk}
&\lVert x_{t+1}\rVert = \lVert \pc_t^{T-t}(x_t)_{y_1} - (x_{t+1}^{c'} - x_{t+1})\rVert \nonumber  \\
&\leq \sum_{m=0}^{t+k-T-1}\lVert \pc_{t-m}^{T-t+m}(x_{t-m})_{y_{m+1}} - \pc_{t-m-1}^{T-t+m+1}(x_{t-m-1})_{y_{m+2}}\rVert \nonumber \\
&+ \lVert \pc_{T-k}^k(x_{T-k})_{y_{t+k-T+1}}\rVert + L_N(\lVert x_{t}\rVert + D_k)\delta^{\kappa} \nonumber \\
&= \sum_{m=0}^{t+k-T-1}\lVert \pc_{t-m}^{T-t+m}(x_{t-m})_{y_{m+1}} - \pc_{t-m}^{T-t+m}(x_{t-m}^{c'})_{y_{m+1}}\rVert \nonumber \\
&+ \lVert \pc_{T-k}^k(x_{T-k})_{y_{t+k-T+1}}\rVert + L_N(\lVert x_{t}\rVert + D_k)\delta^{\kappa} \nonumber \\ 
&\leq \sum_{m=0}^{t+k-T-1}C\delta^{m+1}\lVert x_{t-m} -x_{t-m}^{c'}\rVert + C\delta^{t+k-T+1}\lVert x_{T-k}\rVert \nonumber \\ &+ \frac{2C}{1-\delta}D_k + L_N(\lVert x_t\rVert + D_k)\delta^{\kappa} \nonumber \\
&\leq \sum_{m=0}^{t+k-T}2CL_N\lVert x_{t-m}\rVert \delta^{\kappa+m} + C\delta^{t+k-T+1}\lVert x_{T-k}\rVert + \frac{4CL_N}{1-\delta}D_k,\end{align}}
where $x_{t-m}^{c'} := \pc_{t-m-1}^{T-t+m+1}(x_{t-m-1})_{y_1}$. The first inequality is obtained via telescoping $\pc_t^{T-t}(x_t)_{y_1}$ and applying the following upper bound for $t-m-1 \geq T-k$, {\small \begin{align}\label{eq:29like}\lVert x_{t-m}^{c'} - x_{t-m}\rVert \nonumber &= \left\lVert B\left(\pc_{t-m-1}^{T-t+m+1}(x_{t-m-1})_{v_0} - u_{t-m-1})\right)\right\rVert \nonumber \\
&\leq L_N(\lVert x_{t-m-1}\rVert + D_k)\delta^{\kappa}\end{align} }
where $u_{t-m-1} := \left(\dpc_{t-m-1}^{T-t+m+1}(x_{t-m-1},\knbhd[i])_{v_0[i]}\right)_{i \in \mathcal{V}}$ and we've applied Theorem \ref{thm:diffdec} to obtain the inequality \eqref{eq:29like} which resembles the result in \eqref{eq:cntrlbnd}.
The second inequality in \eqref{eq:geqTminusk} is obtained through applying Corollary \ref{cor:lpsz} twice; the final inequality is then obtained by applying the result in \eqref{eq:29like} and combining terms. This completes the proof of Theorem \ref{thm:statenorm} \qed

\subsection{Proofs of intermediary results in Step 3}\label{subsec:app_3}

\subsubsection{Complete Proof of Theorem \ref{thm:statediffbnd}} For the case $t \geq T-k$, we have that {\footnotesize
\begin{align}&\lVert x_{t+1} - x_{t+1}^*\rVert = \lVert x_{t+1} - \pc_0^T(x_0)_{y_{t+1}}\rVert \nonumber \\ &\leq \lVert x_{t+1} - \pc_t^{T-t}(x_t)_{y_1}\rVert \nonumber \\ &+ \sum_{m=0}^{t-1}\lVert \pc_{t-m}^{T-t+m}(x_{t-m})_{y_{m+1}} - \pc_{t-m-1}^{T-t+m+1}(x_{t-m-1})_{y_{m+2}}\rVert \nonumber \\ 
&\leq L_N(\lVert x_t\rVert + D_k)\delta^{\kappa} +\sum_{m=0}^{t-1}C\delta^{m+1}\lVert x_{t-m} - x_{t-m}^{c'}\rVert \nonumber \\
&\leq L_N(\lVert x_t\rVert + D_k)\delta^{\kappa} \nonumber \\ &+ \sum_{m=0}^{t-1}C\delta^{m+1}\bigg(\lVert x_{t-m} - x_{t-m}^{\tilde{c}} \rVert + \lVert x_{t-m}^{\tilde{c}} - x_{t-m}^{c'}\rVert \bigg) \nonumber \\ &\leq \sum_{m=0}^{t}CL_N(\lVert x_{t-m}\rVert + D_k)\delta^{\kappa+m} \nonumber \\ &+ \sum_{m=t+k-T}^{t-1}\frac{4C^3}{\delta(1-\delta)^2}\delta^{m+1}\left(\delta^{2k}\lVert x_{t-m-1}\rVert + \delta^kD_k)\right)\end{align} }where in the first inequality we telescope $\pc_0^T(x_0)_{y_{t+1}}$. In the second inequality, we apply the result in \eqref{eq:29like} 
and Corollary \ref{cor:lpsz} on the terms in the sum, denoting $x_{t-m}^{c'} := \pc_{t-m-1}^{T-t+m+1}(x_{t-m-1})_{y_1}$. In the third inequality, we apply the triangle inequality and denote $x_{t-m}^{\tilde{c}} := \pc_{t-m-1}^{\min(k,T-t+m+1)}(x_{t-m-1})_{y_1}$. Finally, in the last inequality, we apply \eqref{eq:cntrlbnd} and \eqref{eq:29like} to the terms $\lVert x_{t-m} - x_{t-m}^{\bar{c}}\rVert$, 
and the result in \eqref{eq:Tminustbnd} to the terms $\lVert x_{t-m}^{\bar{c}} - x_{t-m}^{c'}\rVert$. This completes the proof of Theorem \ref{thm:statediffbnd}. \qed 

\subsubsection{Proof of Lemma \ref{lem:onestepterm}} Consider the following optimization problem:
\begin{align}\label{eq:onestepopt}
    \argmin_v \: &c_{t+1}(v) \nonumber \\
    \mathrm{s.t.} \: \: &x_{t+1} = Ax_t + Bv + w_t.
\end{align}
Clearly, the $v$ produced from \eqref{eq:onestepopt} is the same as the one produced from the ``one-step'' terminal state problem $\pct_t^1(x_t,x_{t+1})_{v_0}$. We can then split $v$ such that $v = v_y + v_z$ where $v_y \in \mathrm{Col}(B^\top)$ and $v_z \in \mathrm{Null}(B)$. In particular, $v_y$ must be uniquely determined as
\begin{equation}
    v_y = B^\dagger(x_{t+1} - Ax_t - w_t),
\end{equation}
since $x_{t+1} - Ax_t - w_t$ is in the image of $B$ and for any $v \in \mathbb{R}^{m}$, {\small
$$\lVert Bv - x_{t+1} + Ax_t + w_t\rVert \geq \lVert (BB^\dagger - I)(x_{t+1} - Ax_t - w_t)\rVert = 0,$$}
that is, $B^\dagger(x_{t+1} - Ax_t -w_t)$ is the least-squares solution. Thus, denote $v_z = B_z \omega$ where $B_z$ is the orthonormal basis matrix for $\mathrm{Null}(B)$ and $\omega \in \mathbb{R}^{m-\mathrm{rank}(B)}$. 
Thus, the optimization problem in \eqref{eq:onestepopt} becomes the following unconstrained optimization problem 
\begin{equation}\label{eq:unconstrainedopt}
    \min_{\omega} c_{t+1}(v_y + B_z\omega)
\end{equation}
which achieves its optima at $\omega$ such that 
$B_z^\top\nabla c_{t+1}(v_y + B_z \omega) = 0$. Denote $v' = v_y' + B_z\omega'$  
to be $\pct_t^{1}(x_t',x_{t+1}')_{v_0}$, then comparing the optimality conditions gives us
\begin{align}\label{eq:optconds}
    &B_z^\top(\nabla c_{t+1}(v_y + B_z\omega) - \nabla c_{t+1}(v_y' + B_z\omega')) \nonumber \\
    &= B_z^\top G_c(v_y-v_y' + B_z(\omega-\omega')) = 0 \nonumber \\
    &\iff B_z^\top G_c B_z(\omega-\omega') = -B_z^\top G_c(v_y-v_y'),
\end{align}
where we've applied Lemma \ref{lem:grad2mat} in the first equality. Since $B_z$ is full rank and orthonormal, we have that 
\begin{equation}\label{eq:omegbound}\lVert \omega - \omega'\rVert \leq \frac{L}{\mu}\lVert v_y - v_y'\rVert\end{equation}
by $L$-smoothness and $\mu$-strong convexity of $c_{t+1}(\cdot)$. Finally, we have that
\begin{align}
    \lVert v- v' \rVert^2 &= \lVert v_y-v_y'\rVert^2 + \lVert B_z(\omega-\omega')\rVert^2 \nonumber \\
    &\leq \frac{2L^2}{\mu^2}\lVert v_y-v_y'\rVert^2 \nonumber \\
    &\leq \frac{4L^6}{\mu^2}\left(\lVert x_t-x_t'\rVert^2 + \lVert x_{t+1}-x_{t+1}'\rVert^2\right),
\end{align}
where the first inequality comes from $B_z$ being orthonormal and \eqref{eq:omegbound}. Then the second inequality comes from the bound
\begin{align}\label{eq:rowspacebnd}\lVert v_y - v_y'\rVert &= \lVert B^\dagger(x_{t+1} - x_{t+1}' - A(x_t-x_t'))\rVert \nonumber \\
&\leq L^2(\lVert x_{t} - x_t'\rVert + \lVert x_{t+1} - x_{t+1}'\rVert)\end{align}
and then applying the parallelogram identity. $\Gamma^2$ from Theorem \ref{thm:diffdec} is of the same magnitude as $\frac{4L^6}{\mu^2}$, so we can take $C^2 \approx \frac{4L^6}{\mu^2}$ which finishes the proof. \qed 

\subsection{Proof for Uncertain Forecast Case}\label{subsec:proof_uncertain}
The main difference in proving Theorem \ref{thm:uncertaintyresult} versus proving Theorem \ref{thm:perf_guarantees} is (i) that we require a new perturbation bound on $\pc$ wrt to the parameters $\theta$, and (ii) the regret analysis, which is inspired by the pipeline theorem (Theorem 3.3) in \cite{Lin2022bounded}. The proof of the ISS bound in Theorem \ref{thm:uncertaintyresult} is only a slightly modified version of the one in Theorem \ref{thm:perf_guarantees}. Thus, we begin with the following:

\textbf{Proof of Perturbation Bound on Predictions:} Before presenting the Proofs for the ISS and Regret bounds in Theorem \ref{thm:uncertaintyresult}, we derive an important perturbation bound which is critical to proving Theorem \ref{thm:uncertaintyresult}. 
\begin{theorem}\label{thm:predpertbnd}
    Let $\theta = \theta_{t:t+k}$ and $\theta' = \theta'_{t:t+k}$ be two sets of parameters. Under Assumptions \ref{assump:sys}, \ref{assump:costs}, and \ref{assump:uncertainties}, we have the following perturbation bound
    \begin{multline}\label{eq:predpertbnd}
        \lVert \pc_t^k(x,\theta)_{y_n/v_n} - \pc_t^k(x,\theta')_{y_n/v_n}\rVert \\ \leq \Upsilon\bigg(\sum_{m=0}^{k}\rho^{\lvert n-m\rvert}\lVert \theta_m-\theta_m'\rVert \\ + \lVert x\rVert \sum_{m=0}^{k}\rho^{m+\lvert n-m\rvert}\lVert \theta_m - \theta'_m\rVert\bigg)
    \end{multline}
    where $\Upsilon = \alpha L\left(1+\frac{2\alpha D}{1-\rho}\right)$, and $\alpha$ and $\rho$ are as in Lemma \ref{lem:spatdec}
\end{theorem}
An immediate corollary of Theorem \ref{thm:predpertbnd} is
\begin{corollary}\label{cor:uncnextstepbnd}
    Let $x_{t+1} = Ax_t + B\pc_t^k(x_t,\theta)_{v_0}$ and $x_{t+1}^u = Ax_t + B\pc_t^k(x_t,\theta^*)_{v_0}$
    \begin{equation}\label{eq:uncxt+1xtbnd}
        \lVert x_{t+1} - x_{t+1}^u \rVert \leq L\Upsilon\left(\sum_{m=0}^{k}\rho^m\left(1+\lVert x_t\rVert\rho^m\right)\phi_{t+m|t}\right)
    \end{equation}
\end{corollary}

\begin{proof}[Proof of Theorem \ref{thm:predpertbnd}]

First, let us observe the difference in the KKT conditions between the two optimization problems $\pc_t^k(x,\theta)$ and $\pc_t^k(x,\theta')$:

$$\begin{bmatrix}
    \nabla_z \tilde{f}(z^\theta,\theta) + J^\top\lambda^\theta \\ J z^\theta \end{bmatrix} = \underbrace{\begin{bmatrix}
        0 \\ \begin{bmatrix}
            x \\ 
            \zeta(\theta)
        \end{bmatrix}
    \end{bmatrix}}_{b^\theta},$$ and, $$ \begin{bmatrix}
    \nabla_z \tilde{f}(z^{\theta'},\theta') + J^\top\lambda^{\theta'} \\ J z^{\theta'} \end{bmatrix} = \begin{bmatrix}
        0 \\ \begin{bmatrix}
            x \\ 
            \zeta(\theta')
        \end{bmatrix}
    \end{bmatrix}
$$
where $(z^\theta,\lambda^\theta)$ is the primal and dual solution vector of $\pc_t^k(x_t,\theta)$ and $(z^{\theta'}, \lambda^{\theta'})$ is that of $\pc_t^k(x_t,\theta')$. The cost $\tilde{f}$ is defined as
$$\tilde{f}(z^\theta,\theta) := \sum_{\tau=0}^{\ell-1}f_{t+\tau}(y_\tau^{\theta_\tau},\theta_\tau) + c_{t+\tau+1}(v_{\tau}^{\theta_\tau},\theta_{\tau}) + F(y_{\ell}^{\theta_{\ell}}, \theta_{\ell}).$$ Taking the difference of the two systems of equations from above gives us
\begin{align*}&\begin{bmatrix}
    \nabla_z\tilde{f}(z^\theta, \theta) - \nabla_z\tilde{f}(z^{\theta'},\theta') + J^\top(\lambda^{\theta}-\lambda^{\theta'})\\ J(z^\theta - z^{\theta'})
\end{bmatrix} \\ &= \begin{bmatrix}
    0 \\ \begin{bmatrix}
        0 \\ \zeta(\theta) - \zeta(\theta')
    \end{bmatrix}
\end{bmatrix}.\end{align*}
By Lemma \ref{lem:grad2mat}, the cost gradients can be written as 
$$\nabla_z\tilde{f}(z^\theta,\theta) = G(z^\theta,0;\theta)z^\theta$$ and $$ \nabla_z\tilde{f}(z^{\theta^*},\theta^*) = G(z^{\theta'},0;\theta')z^{\theta'}.$$
Thus, we can rewrite the difference from before as the following new system of equations:
\begin{multline*}\underbrace{\begin{bmatrix}
    G(\theta) & J^\top \\ J
\end{bmatrix}}_{H^{\theta}}\underbrace{\begin{bmatrix}
    z^\theta \\ \lambda^\theta
\end{bmatrix}}_{q^\theta} - \underbrace{\begin{bmatrix}
    G(\theta') & J^\top \\ J
\end{bmatrix}}_{H^{\theta'}}\underbrace{\begin{bmatrix}
    z^{\theta'} \\ \lambda^{\theta'}
\end{bmatrix}}_{q^{\theta'}} \\ = \underbrace{\begin{bmatrix}
    0 \\ \begin{bmatrix}
        0 \\ \zeta(\theta)-\zeta(\theta')
    \end{bmatrix}
\end{bmatrix}}_{b^{+}}.\end{multline*}
Observe that the difference between vectors $q^\theta$ and $q^{\theta'}$ is:
$$q^{\theta}-q^{\theta'} = (H^{\theta'})^{-1}\left(b^+ - (H^{\theta}- H^{\theta'})(H^{\theta})^{-1}b^{\theta}\right),$$
where the inverses of $H^{\theta'}$ and $H^{\theta}$  have block entries decaying with constants $\alpha$ and $\rho$ due to Lemma \ref{lem:spatdec}. Thus, for some $n \leq k$, 
$$\begin{aligned}&\lVert q^{\theta}_n - q^{\theta'}_n\rVert =\\  &\left\lVert \sum_{m = 0}^{k}(H^{\theta'})_{nm}^{-1}\left(b^+_m - \sum_{l=0}^{k}(H^\theta-H^{\theta'})_{ml}\sum_{p=0}^{k}(H^\theta)^{-1}_{lp}b^\theta_{p}\right) \right\rVert \\
&\leq  \sum_{m=0}^{k}\alpha\rho^{\lvert n-m\rvert}\left(\lVert b^+_m\rVert  + L\lVert \theta_m -\theta_m^\prime\rVert \sum_{p=0}^{k}\alpha\rho^{\lvert p-m\rvert}\lVert b_p^\theta\rVert \right)\\
&\leq  \alpha L\Bigg(\sum_{m=1}^{k}\rho^{\lvert n-m\rvert}\lVert \theta_{m-1} - \theta^\prime_{m-1}\rVert  + \alpha\sum_{m=0}^{k}\rho^{\lvert n-m\rvert}\\ &\qquad \cdot\lVert\theta_m-\theta^\prime_m\rVert\left(\rho^m\lVert x\rVert + \sum_{p=1}^{k}\rho^{\lvert p-m\rvert}\lVert \zeta_{p-1}(\theta_{p-1})\rVert \right) \Bigg)\\ 
&\leq \alpha L\Bigg(\left(1 + \frac{2\alpha D}{1-\rho}\right)\sum_{m=0}^{k}\rho^{\lvert n-m\rvert}\lVert \theta_m-\theta_m^\prime\rVert \\ &\qquad\qquad\quad  + \alpha\lVert x\rVert \sum_{m=0}^{k}\rho^{m+\lvert n-m\rvert}\lVert \theta_m-\theta_m^\prime\rVert  \Bigg)\end{aligned}$$
where in the first inequality we've applied (i) Assumption \ref{assump:uncertainties}, (ii) the exponential decay bounds on the norms of $(H^{\theta'})^{-1}$ and $(H^\theta)^{-1}$, and (iii) the fact that $H^\theta - H^{\theta'}$ is block diagonal. In the second inequality we apply Assumption \ref{assump:uncertainties} and upper bound $b_p^\theta$ by its components. In the last line we apply Assumption \ref{assump:uncertainties} again, and the fact that the disturbances $w_t$ are bounded to achieve the desired upper bound \end{proof}

\textbf{Proof of ISS:} The main difference from the proof of Theorem \ref{thm:statenorm} is that when we consider upper bounding the norm of $\lVert x_{t+1}\rVert$, we must incur an error term $\lVert x_{t+1}^s - x_{t+1}^p\rVert$ from the predictions where $x_{t+1}^s = Ax_t+B\pc_t^k(x_t,\theta_{t:t+k})_{v_0}$ and $x_{t+1}^u = Ax_t + B\pc_t^{k}(x_t,\theta_{t:t+k}^*)_{v_0}$. We can upper bound this extra error term using Corollary \ref{cor:uncnextstepbnd}; following the same telescoping technique in the proof of Theorem \ref{thm:statenorm}, we obtain the 3 following upper bounds on the state for $uDTPC_k$: 

\noindent\textbf{For $t \leq k-1$)}
\begin{multline}\label{eq:uncertxtbnd_1}
    \lVert x_{t+1}\rVert \leq \sum_{m=0}^{t}\bigg(LC^2\sum_{n=0}^{k}(\phi_{t-m+n|t-m}\delta^{2n+m}) \\ + C^2N\del^{\kappa+m}+ 2C^3\delta^{2k-2-m}\bigg)\lVert x_{t-m}\rVert + C\lVert x_0\rVert + W.
\end{multline}
\textbf{For $k \leq t \leq T-k-1$)}
\begin{multline}\label{eq:uncertxtbnd_2}
    \lVert x_{t+1}\rVert \leq \sum_{m=0}^{k-1}\bigg(LC^2\sum_{n=0}^{k}(\phi_{t-m+n|t-m}\delta^{2n+m}) \\ + C^2N\del^{\kappa+m}+ 2C^3\delta^{2k-2-m}\bigg)\lVert x_{t-m}\rVert + W.
\end{multline}
\textbf{For $t \geq T-k$)}
\begin{multline}\label{eq:uncertxtbnd_3}
    \lVert x_{t+1}\rVert \leq \sum_{m=0}^{t+k-T}\bigg(LC^2\sum_{n=0}^{T-t}(\phi_{t-m+n|t-m}\delta^{2n+m}) \\ + C^2N\del^{\kappa+m}\bigg)\lVert x_{t-m}\rVert + C\delta^{t+k-T+1}\lVert x_{T-k}\rVert + W.
\end{multline}
Where we remind ourselves that $C := \max(\Omega,\Gamma,\Upsilon)$. Following the remainder of the ISS proof in the certainty setting, the choice of $k$ and $\kappa$ along with the conditions on the predictions in Theorem \ref{thm:uncertaintyresult} guarantees the ISS bound in \eqref{eq:uncertISS}. \qed

\textbf{Proof of Regret:} From Lemma 3.2 from \cite{Lin2022bounded}, the regret of $uDTPC_k$ is upper bounded by
\begin{multline}\label{eq:uncerrorreg}
    \cost(uDTPC_k) - \cost(OPT) \\ \leq \sqrt{C_R\cost(OPT)\sum_{t=0}^{T-1}e_t^2} + C_R\sum_{t=0}^{T-1}e_t^2,
\end{multline}
where $C_R := \frac{L}{2}\left(1+\frac{2L\alpha^2}{1-\rho}\right)\left(1+\frac{1}{1-\rho}\right)$, $L$ and $\alpha$ are as in Lemma \ref{lem:spatdec}, and $e_t$ is the control error
\begin{equation}\label{eq:unc_cntrlerr}
    e_t := \lVert u_t - \pc_t^{T-t}(x_t,\theta^*_{t:T};f_T)_{v_0}\rVert,
\end{equation}
where $u_t := (\dpc_t^{k'}(x_t,\theta_{t:t+k'|t};F)_{v_0[i]})_{i\in\calV}$ is the control input generated by $uDTPC_k$ and $k' = \min(k,T-t)$. It is clear that all we need now is an upper bound on $e_t^2$.
\begin{theorem}\label{thm:uncerrbnd}
    Under the assumptions of Theorem \ref{thm:uncertaintyresult}, we have that the control error $e_t$ is bounded as 
    \begin{multline}\label{eq:unc_cntrlerrbnd}
        e_t  \leq \Upsilon \sum_{n=0}^{k^\prime}\left(\rho^{n}\left(1+\rho^n\lVert x_t\rVert \right)\phi_{t+n|t}\right) \\ +  \Gamma N(\lVert x_t\rVert + D_k)\delta_\calS^{\kappa}  + 2 \Omega^2\delta_\calT ^{k}\left(\delta_\calT^{k}\lVert x_t\rVert + \frac{2D}{1-\delta_\calT}\right),
    \end{multline}
    and
    \begin{multline}\label{eq:unc_cntrlerrsqbnd}
        e_t^2 \leq \left(\Upsilon \left(\frac{1}{1-\rho}+\frac{R_0}{1-\rho^2}\right)+C_\kappa + C_k\right)\\ \cdot \left(\sum_{n=0}^{k}\Upsilon (1+R_0\rho^n)\rho^n\phi_{t+n\vert t}^2 + C_\kappa\del_\calS^{2\kappa} + C_k\del_\calT^{2k}\right),
    \end{multline}
    where $R_0 := \frac{C^2}{\xi^2}\lVert x_0\rVert + \frac{2CW}{\xi}$ from the ISS bound in \eqref{eq:uncertISS}, $C_\kappa := \Gamma N(R_0 + D_k)$, and $C_k := 2\Omega^2\left(R_0 + \frac{2D}{1-\del_\calT}\right).$
\end{theorem}
The regret bound in \eqref{eq:uncert_reg_bnd} immediately follows from the upper bound in \eqref{eq:unc_cntrlerrsqbnd}. \qed 

\begin{proof}[Proof of Theorem \ref{thm:uncerrbnd}]
    Let $u_t^s := \pc_t^{k'}(x_t,\theta_{t:t+k}')_{v_0}$ and $u_t^p := \pc_t^{k'}(x_t,\theta_{t:t+k'}^*)_{v_0}$. We can estimate an upper bound for the control error by splitting it into the 3 following terms:
    \begin{equation}\label{eq:uncerrsplit}
        e_t \leq \lVert u_t - u_t^s\rVert + \lVert u_t^s-u_t^p\rVert + \lVert u_t^p - \pc_t^{T-t}(x_t,\theta_{t:T}^*)_{v_0}\rVert.
    \end{equation}
    Note that the third term vanishes for $t \geq T-k$. Furthermore, the third term can expressed as
    $$\lVert \pct_t^k(x_t,\pc_t^k(x_t;F)_{y_k})_{v_0} - \pct_t^k(x_t,\pc_t^{T-t}(x_t;f_T)_{y_k})_{v_0}\rVert$$
    from the principle of optimality. Thus, we can apply Lemma \ref{lem:termlpz} to this expression to obtain the upper bound:
    \begin{equation}\label{eq:temperrbnd}\Omega\del_\calT^k\lVert \pc_t^k(x_t)_{y_k} - \pc_t^{T-t}(x_t)_{y_k}\rVert.\end{equation}
    Applying Corollary \ref{cor:lpsz} to equation \eqref{eq:temperrbnd}, and applying Theorem \ref{thm:diffdec} and Corollary \ref{cor:uncnextstepbnd} to the first two terms in equation \eqref{eq:uncerrsplit} gives us the desired upper bound in \eqref{eq:unc_cntrlerrbnd}. To obtain equation \eqref{eq:unc_cntrlerrsqbnd}, we apply cauchy schwarz and upper bound the state as: $\lVert x_t\rVert \leq R_0$ by using the ISS bound in \eqref{eq:uncertISS}. Upper bounding the resulting expression by geometric sums finishes the proof. 
\end{proof}

\subsection{Simulation Details and Setup}\label{subsec:app_sim}

Let $T_t$ be the vector containing the temperature of each zone at time $t$ and $U_t$ the respective integrators. Let $u_t$ be the vector containing the manipulated heat generation/absorption of each zone at time $t$, and let $w_t$ be the disturbances at time $t$. then the Euler-discretized dynamics of the system for a sampling time $t_s = 1$ (seconds) are
\begin{equation}\label{eq:hvacdyn}
    \begin{bmatrix}
        U_{t+1} \\ T_{t+1}
    \end{bmatrix} = \begin{bmatrix}
        I & t_s I \\ 0 & I - t_sL
    \end{bmatrix}\begin{bmatrix}
        U_t \\ T_t
    \end{bmatrix} + \begin{bmatrix}
        0 \\ 0.5 t_s I
    \end{bmatrix}u_t + w_t,
\end{equation}
where $L[i,j] = k_{ij} = 0.05$ for $i,j \in \mathcal{V}$ is the weighted graph Laplacian and each $k_{ij}$ corresponds to the degree of heat exchange between zones $i$ and $j$. The disturbances are normally distributed as $w_t \sim \mathcal{N}(0, 25I)$, a multivariate Gaussian random variable with mean $0$ and covariance matrix $25I$. As for the costs, we set $f_{t}(x_t) := \frac{1}{2}x_t^\top Q x_t$ for constant $Q = I$ and all $t$, $F(x_t) = \frac{1}{2}x_t^\top Q_F x_t$ for constant $Q_F = 10I$ and finally we have $c_{t+1}(u_t) = \frac{1}{2}u_t^\top R_t u_t$ for time varying $R_t = \mathrm{diag}(5\lvert \mathbf{Z}\rvert) + I$ where the $\mathrm{diag}(\cdot)$ operator creates a diagonal matrix whose entries correspond to its input vector's, and $\mathbf{Z}$ is the standard multivariate Gaussian with mean $0$ and covariance matrix $I$. The system assumptions of $(A,B)$ in Assumption \ref{assump:sys} are further detailed in Section 5 of \cite{Sungho2023nearopt}.}{}
\end{document}